\documentclass[12pt,english]{smfart} 
\usepackage{amssymb} 
\usepackage{amsfonts} 
\usepackage{physics} 
\usepackage{esint} 
\usepackage{braket} 
\usepackage{enumitem}
\usepackage[toc,page]{appendix} 
\usepackage{tikz}
\usepackage{xcolor}
\usepackage{graphicx,epsfig,subfigure,psfrag}
 \usepackage{listings}

  \usepackage[usenames,dvipsnames]{pstricks}
 \usepackage{pstricks-add}
 \usepackage{epsfig}

\newtheorem{theorem}{Theorem}[section]
\newtheorem{proposition}[theorem]{Proposition}

\newtheorem{corollary}[theorem]{Corollary}
\newtheorem{definition}[theorem]{Definition}
\theoremstyle{remark}
\newtheorem{remark}[theorem]{Remark}

\newcommand{\R}{\mathbb{R}}
\newcommand{\dist}{\rm{dist}}
\newcommand{\spt}{\mathrm{spt}}
\newcommand{\HH}{\mathcal{H}} 
\newcommand{\N}{\mathbb{N}} 

\newcommand{\Cc}{\mathcal{C}}

\newcommand{\Hm}{{\bf H}}
\newcommand{\Vect}{{\rm Vect}}
 
\author[]{Jean-Fran\c{c}ois Grosjean, 
Antoine Lemenant
 and Rémy Mougenot}

\address{Universit\'e de Lorraine, CNRS, Institut Elie Cartan de Lorraine, BP 70239 54506 Vand\oe uvre-l\`es-Nancy Cedex, France}

\email{jean-francois.grosjean@univ-lorraine.fr, antoine.lemenant@univ-lorraine.fr, remy.mougenot@univ-lorraine.fr}

\date{\today}

\title{Reilly inequality for Varifolds}

\begin{document}

\begin{abstract} The famous Reilly inequality gives an upper bound for the first eigenvalue of the Laplacian defined on compact submanifolds of the Euclidean space in terms of the $L^2$-norm of the mean curvature vector. In this paper, we generalize this inequality in a Varifold context. In particular we generalize it for the class of $H(2)$ varifolds and for polygons and we analyse the equality case.  
\end{abstract}

\maketitle
\setcounter{tocdepth}{2}
\tableofcontents

 \newpage
\section{Introduction}

In this work, we extend the well-known Reilly inequality, originally established for smooth submanifolds of Euclidean space $\R^{n+1}$ (see \cite{Rei}), to the more general setting of varifolds. Varifolds provide a broad generalization of differentiable smooth submanifolds, of which they are a special case.

We begin by recalling the classical Reilly inequality. Let $M^m$ be a compact, connected submanifold of $\R^N$ of class $\Cc^2$, ensuring that the mean curvature is well-defined. The mean curvature vector $\Hm$ is given by the trace of the second fundamental form associated with the metric $g$ induced by the canonical metric of $\R^N$. Denote by $\lambda_1^M$ the first nonzero eigenvalue of the Laplace–Beltrami operator on $M$ with respect to $g$. Reilly's inequality (\cite{Rei}) states that
\begin{align*}
\lambda_1^M\leqslant\frac{1}{mv_M}\int_M|\Hm|^2dv 
\end{align*}
where $v_M$, $|\Hm|$, and $dv$ represent, respectively, the volume of $M$, the Euclidean norm of the mean curvature vector $\Hm$, and the volume element with respect to $g$. Furthermore, this inequality is sharp, as equality holds if and only if $M$ is a minimal submanifold of a sphere of radius $\sqrt{m/\lambda_1^M}$. In the particular case where $M$ is a hypersurface, i.e., $m=N-1$, equality holds if and only if $M$ is a geodesic sphere of radius $\sqrt{m/\lambda_1^M}$.

This inequality has been extended to submanifolds of other ambient spaces, such as the sphere $\mathbb{S}^{N}$ or the hyperbolic space $\mathbb{H}^{N}$, where $|\Hm|^2$ is replaced by $|\Hm|^2+c$, with $c=1$ for $\mathbb{S}^{N}$ and $c=-1$ for $\mathbb{H}^{N}$ (see \cite{ElsIli1,Hei}, and more recently \cite{NiuXu}, where the authors resolved an open problem posed by Heintze in 1988). In these ambient spaces, the equality case is analogous to that in the Euclidean setting.

Numerous studies have been conducted on this inequality. For instance, $\lambda_1^M$ has been shown to be bounded by other curvature quantities, such as the scalar curvature (\cite{AleDocRos,Gro1}). Additionally, Reilly's inequality has been extended to the first eigenvalue of other elliptic operators, including the Schrödinger operator, the $p$-Laplacian, and the Steklov eigenvalue (see, e.g., \cite{CheGui,CheWan,ElsIli2,Gro0,Gro2,IliMak,RotUpa}).

It is also worth noting that the stability of the equality case for hypersurfaces in Reilly’s inequality has recently been investigated in \cite{AubGro1} and \cite{AubGro2}. We now proceed to state Reilly's inequality in the more general framework of varifolds.


The first category of Varifolds    that we consider are the ones in the class $H(2)$: they are rectifiable $m$-varifolds  with locally bounded  variation such that  $\|\delta V\|$ is absolutely continuous with respect to $\|V\|$, with density ${\bf H}$  belonging to $L^2(\mathbb{R}^N,d\|V\|)$ (we refer to Section \ref{SectionVarifolds} for a more precise definition).

If $V\in H(2)$ then we define $\lambda_1(V)$ as follows:

 \begin{eqnarray} \lambda_1:=\inf_{u \in C^\infty(\mathbb{R}^N) \setminus \{0\} \text{ s.t. } \displaystyle\int  u \; d\|V\|=0} \dfrac{\displaystyle\int |\nabla_P u(x)|^2 \;d V(x,P)}{\displaystyle \int_{\mathbb{R}^N} u^2\;d\|V\|}.  \notag
 \end{eqnarray}

If $V$ is the varifold associated to a smooth surface then $\lambda_1$ coincides with the usual first eigenvalue of the Laplacian.

Our first main result if the following one.

\medskip

\medskip\begin{theorem}[Reilly inequality in $H(2)$] \label{reilly1} Let $K$ be a bounded rectifiable set and assume that $V=g\HH^m\mid_{K}\otimes\delta_{TK}\in H(2)$ is an $m$-Varifold with finite mass  $M:= V(\mathbb{R}^N\times G(N,m))$. Then
 \begin{eqnarray}
 \lambda_1 \leqslant \frac{\|{\bf H} \|^2_{L^2(d\|V\|)}}{mM}. \label{ineqR}
\end{eqnarray} 
 \end{theorem}
 
In particular, Theorem \ref{reilly1} encompasses the standard Reilly inequality for smooth surfaces as a special case. However, the proof must be adapted to the varifold setting and does not follow exactly the same approach as in the smooth case.

Furthermore, as in the case of smooth surfaces, we deduce that for codimension 1 varifolds, equality holds if and only if the varifold is a multiple of a sphere. In higher codimensions, equality occurs for stationary varifolds in the sphere. The precise result is stated below.
 
\medskip\begin{theorem}\label{equalityH2} Let $K$ be a bounded rectifiable set and assume that $V=g\HH^m\mid_{K}\otimes\delta_{TK}$ is an $H(2)$-integral $m$-varifold with finite mass  $M= V(\mathbb{R}^N\times G(N,m))\not = 0$ which satisfies the equality case in the Reilly inequality, 
 \begin{eqnarray}
mM \lambda_1 = \|{\bf H} \|^2_{L^2(d\|V\|)}. \label{ineqReq}
\end{eqnarray}
Then up to translate at the origin ${\rm supp}(\|V\|) \subset \partial B\left(0,R\right)$ with $R=\sqrt{m/\lambda_1}$ and the varifold $V^S$ of the sphere $S_{R}=\partial B\left(0,R\right)$ defined by $V^S=g\HH^m\mid_{L}\otimes\delta_{TL}$ where $L={\rm supp}(\|V\|)$ is $H(2)$ and the mean curvature of $V^S$ in $S_R$ is vanishing. Moreover if $m=N-1$, then $V$ is the  $N-1$-dimensional sphere, with constant multiplicity. More precisely, up to translate at the origin  $V=n\mathcal{H}^{N-1}|_{S_R}\otimes\delta_{TS_R}$  and $n\in \mathbb{N}^*$. 
\end{theorem}

In a second part of the paper we  focus on polygons, which are special cases of varifolds. In this class we obtain a Reilly type inequality but which  requires a new definition of spectrum,  as we will see in Section \ref{polygonSection}.  More precisely, for a polygon  $K_A=\bigcup_{i=1}^n [A_i,A_{i+1}]$ with vertices $A_i\in \R^N$ (where $A_{n+1}=A_1$)  we define:

 \begin{eqnarray}
 \lambda_1^A:=\inf_{u \in C^\infty(\mathbb{R}^N) \setminus \{0\} \text{ s.t. } \int_{K_A} u \; d \mathcal{H}^1=0} \frac{\int_{K_A} |\nabla_P u(x)|^2 \;d \mathcal{H}^1}{\sum_{i=1}^{n} u^2(A_i)}, \label{lambdaPintro}
 \end{eqnarray}
where $\nabla_P u$ is the tangential gradient of $u$, defined $\mathcal{H}^1$-a.e. on the 1-rectifiable set $K_A$. We also consider the rectifiable $1$-varifold $V_A$  associated to this polygon, which allows to define the mean curvator vector measure. This vector is zero on every edge, and is the Dirac measure $\Hm_i \delta_{A_i}$ on each vertex, where  (see Section \ref{polygon})
$$\Hm_i= \left(\frac{A_{i}-A_{i-1}}{|A_i-A_{i-1}|}+ \frac{A_{i}-A_{i+1}}{|A_i-A_{i+1}|}\right).$$
As a matter of fact, with does definitions we establish a nice Reilly inequality which is stated below.

\medskip
\medskip\begin{theorem}[Reilly inequality for polygons]\label{ReillyPolygon} Let  $V_A$ be the rectifiable $1$-varifold associated to a polygon $A=\{A_i\}_{1\leqslant i \leqslant k+1}$ in $\R^N$ as defined in the Section \ref{polygon}, supported on the set $K_A=\bigcup_{1\leqslant i\leqslant k}[A_i,A_{i+1}]$.  Then
 \begin{eqnarray}
 \mathcal{H}^1(K_A) \lambda_1^A \leqslant \|{\bf H} \|^2_{L^2(d \mu_A)}=\displaystyle\sum_{i=1}^{n} |\Hm_i|^2, \label{ineqRP}
\end{eqnarray} 
where $\Hm=\displaystyle\sum_{i=1}^{n}\Hm_i\delta_{A_i}$.
\end{theorem}

Next, it is natural to look for the equality case in this context. What is somehow surprising  is that the equality case is not verified exclusively by the regular polygon. We indeed have found three other types of polygon that saturate the Reilly inequality, that are losanges,  ``trapezes" and the ``fake-regular" polygons. Even more unexpectedly, the trapeze and the fake-regular polygon are asymmetric.

\medskip\begin{theorem}\label{Equality-Case} Let $A=\{A_i\}_{1\leqslant i \leqslant n+1}$ be a polygon in $\R^N$ which realizes the equality in the Reilly inequality \eqref{ineqRP}. Then $A$ is a planar regular polygon, or a losange, or a trapeze (as in Definition \ref{definitionT}), or a fake-regular polygon (as defined in Section \ref{fakeregular}).
\end{theorem}

Finally in a last section we consider the case of a star-shaped graphs, for which we also establish a Reilly inequality and analyse the equality case, as being stationary graphs in the sense of varifold theory (see Section \ref{graph}). 

\section{Varifolds}
\label{SectionVarifolds}
A $m$-varifold $V$  in $\mathbb{R}^N$ is a Radon measure on $\mathbb{R}^N\times G(N,m)$, where $G(N,m)$ is the Grassmanian, i.e. the set of all unoriented $m$-dimensional subspaces of $\mathbb{R}^N$.  We also denote by  $\|V\|$ the weight  of $V$, i.e. its canonical projection into $\mathbb{R}^N$ (see for instance \cite{Sim} or \cite{Ton} for a short review on Varifold theory). In other words $\|V\|$ is the Radon measure defined by 
$$\forall \varphi \in C^0_c(\mathbb{R}^N), \quad\|V\|(\varphi)=\int \varphi(x) \;dV(x,S).$$
In this paper the integral over $\mathbb{R}^N\times G(N,m)$ is simply denoted without any domain of integration. Let us denote by $\mathcal{X}(\R^N)$ the space of compactly supported smooth vector fields on $\R^{N}$. For any $\theta\in\mathcal{X}(\R^N)$ and $P \in G(N,m)$ we define the tangential divergence by
$${\rm div}_P(\theta)(x) = {\rm Trace}(\pi_P(D\theta))(x)=\sum_{i=1}^p(D_{e_i}\theta \cdot e_i)(x),$$
where $\pi_P$ is the orthogonal projection onto $P \in G(N,m)$, $(e_i)_{1\leqslant i\leqslant m}$ is an orthonormal basis of $P$, $D$ is the differential operator and $\cdot$ denotes the scalar product on $\R^N$. Note that for any vector field $X=\sum_{i=1}^NX^i\partial_i$ and $\theta=\sum_{i=1}^N\theta^i\partial_i$, we have $$D_X\theta(x)=\displaystyle\sum_{1\leqslant i,j\leqslant N}X^j\partial_j\theta^i\partial_i.$$
where $(\partial_i)_{1\leqslant i\leqslant N}$ is the canonical basis of $\R^N$. The operator $D$ is in fact the Levi-Civita connection of the Euclidean space $\R^N$. The first variation of $V$ is defined for any $\theta\in\mathcal{X}(\R^N)$ by
\begin{equation}\label{First-variation}
\delta V(\theta) = \int  {\rm div}_P( \theta)(x) \;d V(x,P).
\end{equation}
We say that $V$ has a locally bounded first variation  if 
$$\sup \{|(\delta V)(\theta)| \;  : \; \theta \in\mathcal{X}(\R^N)\; ,\; \spt(\theta)\subset U\; ,\; |\theta|\leqslant 1 \}$$
is finite on any bounded set $U$. If this happens then from the vectorial Riesz representation theorem there exists a Radon measure denoted by $\|\delta V\|$ and a  $\|\delta V\|$-measurable vector field $\nu$ on $\R^N$ with compact support such that for all $\theta\in\mathcal{X}(\R^N)$ we have
\begin{equation}\label{Riesz-representation}
\delta V(\theta)=\int_{\R^N}\theta\cdot\nu\; d\|\delta V\|
\end{equation}
and $|\nu(x)|=1$ for $\|\delta V\|$ a.e. $x\in\R^N$. 
Moreover $\|\delta V\|$ encodes the mean curvature of $V$. From the Radon-Nikodym decomposition of $\|\delta V\|$ there exists $f\in L^1_{loc}(\R^N,\R^N,\|V\|)$ such that 
$$\|\delta V\|=f\|V\|+\|\delta V\|_{sing}$$
where $\|\delta V\|_{sing}$ is the singular part. Defining $\Hm(x)=-f(x)\nu(x)$, we have 
$$\delta V(\theta)=-\int_{\R^N}\Hm\cdot\theta \; d\|V\|+\int_{\R^N}\nu\cdot\theta \; d\|\delta V\|_{sing}.$$
The vector field $\Hm$ is called the generalized mean curvature.  

Note that the definition of varifolds can be extended to an $N$-Riemannian manifold $(M,h)$. For any $x\in M$, let us denote $G(T_xM,m)$ the set of all $m$-dimensional linear subspace of $T_xM$ and let us condider $G_m(TM)$ the Grassmann $m$-plane bundle of the tangent bundle $TM$ of $M$. An $m$-dimensional varifold $V$ is a Radon measure over $G_m(TM)$. As in the Euclidean case, the weight of $V$ is defined by 
$$\forall \varphi \in \Cc^0_c(M), \quad \|V\|(\varphi)=\int_{G_m(TM)} \varphi(x) \;dV(x,S).$$
Let us denote by $\mathcal{X}(M)$ the space of compactly supported smooth vector fields on $M$. Then for any $\theta\in\mathcal{X}(M)$, $x\in M$ and $P\in G(T_xM,m)$, the tangential divergence of $\theta$ at $(x,P)$ is given by
$${\rm div}^M_P\theta(x)=\sum_{i=1}^ph(D^M_{e_i}\theta, e_i)(x),$$
where $(e_i)_{1\leqslant i\leqslant m}$ is an orthonormal basis of $P$ and $D^M$ is the Levi-Civita connection of $M$. The first variation  of $V$ is defined for any $\theta\in\mathcal{X}(M)$ by
$$\delta V(\theta) = \int  {\rm div}^M_P( \theta)(x) \;d V(x,P).$$
Similarly to the Euclidean case, we define varifolds with locally bounded variation as well as $\|\delta V\|$ and the generalized mean curvature $\Hm$ and the singular measure $\|\delta V\|_{sing}$.





\subsection{Classes of Varifolds}

A first particular class of  varifolds  are those with  $L^p$ curvature  as defined below. 

\medskip
\begin{definition}[The class $H(p)$] A varifold $V$ with locally bounded  variation belongs to the class $H(p)$ if $\|\delta V\|$ is absolutely continuous with respect to $\|V\|$ and its density ${\bf H}$  belongs to $L^p(\mathbb{R}^N,d\|V\|)$.
\end{definition}

If $V$ belongs to the class $H(p)$, then for any $\theta\in\mathcal{X}(\R^N)$, we have 
$$\delta V(\theta)=-\int_{\R^N}\Hm\cdot\theta \; d\|V\|.$$
Another important class of varifolds is the class of Rectifiable varifolds.

\medskip\begin{definition}[Rectifiable Varifold] An $m$-varifold $V$ is rectifiable if there exists an $m$-rectifiable set $K \subset \R^N$ and a $\mathcal{H}^m$ mesurable function $g:K \to (0,+\infty)$  such that $g\mathcal{H}^m\mid_K$ is a Radon measure and for all $\varphi \in C^0_c(\mathbb{R}^N\times G(N,m))$ we have
 $$\int \varphi \; dV = \int_{K} \varphi(x, T_xK) g(x) \;d\mathcal{H}^m(x) ,$$
 where $T_xK$ is the approximative tangent plane of $K$ at point $x\in K$. We will denote $V=g\mathcal{H}^m\mid_K\otimes\delta_{TK}$.
\end{definition}

\medskip\begin{definition}[Integral Varifold] An integral $m$-varifold $V$ is a rectifiable varifold with integer multiplicity, i.e. $g \in \N$.
\end{definition}

All these definitions are easily generalized in the case of varifolds of a Riemannian manifold. In the sequel we shall need the following result.

\medskip\begin{proposition}[\cite{Ton}, Theorem 1.18]\label{orthogonality} Let $V$ be an integral $m$-varifold in $H(2)$ supported on a rectifiable set $K$. Then for $\mathcal{H}^m-a.e.$ $x \in K$, the approximative tangent plane to $K$ at point $x$ is orthogonal to the generalized mean curvature ${\bf H}(x)$.
\end{proposition}

\subsection{Example of varifold associated to a sphere}\label{polygon1}

Let $\partial B(0,R)$ be a sphere of radius $R>0$ and let $V$ be the varifold defined for $\varphi \in C_c^{0}(\R^N \times G(N,m))$ by 
$$V(\varphi)=\int_{\partial B(0,R)} \varphi(x,P(x)) \; d\mathcal{H}^{N-1}(x),$$
where $P(x)$ is the tangent plane to $\partial B(0,R)$ at point $x$. Then for all $\theta\in C^{1}(\R^N,\R^N)$,
$$\delta V(\theta)=\int_{\partial B(0,R)} {\rm div}_{P(x)}(  \theta)(x) \; d\mathcal{H}^{N-1}(x)=\int_{\partial B(0,R)} {\bf H}\cdot  \theta  \; d\mathcal{H}^{N-1}$$
with ${\bf H}= \nu (N-1)/R$ where $\nu$ is the outwise normal to the sphere.

\subsection{Example of varifold associated to a polygon}\label{polygon}We start by giving a precise definition of a polygon and its associated varifold.
 
\medskip\begin{definition} \label{defPolygon} For $n\geqslant 3$ we say that  $A=\{A_i\}_{1\leqslant i \leqslant n+1}$ is  a polygon with vertices $A_i \in \R^N$ if  $A_{n+1}=A_1$ with no self-intersection points, i.e. we assume that  the open  segment $]A_i,A_{i+1}[$ does not intersect any another segment $]A_j,A_{j+1}[$, and we also assume that each vertex $A_i$ belongs to the two (and only two) consecutive segments $[A_{i-1},A_i]$ and $[A_i,A_{i+1}]$. We then define the rectifiable set $K_A$ defined by the union of the segment joining the vertices $A_i$, namely
 $$K_A:=\bigcup_{1\leqslant i \leqslant n}[A_i , A_{i+1}],$$
and the associated rectifiable varifold $V_A$ with constant multiplicity one  defined through 
 $$\int \varphi \;dV_{A}:=\int_{K} \varphi(x,P(x)) \;d\mathcal{H}^1(x).$$
Here $P(x)$ is the approximate tangent plane thus equals $\textnormal{Vect}\{A_{i+1}-A_{i}\}$ on the segment $[A_i,A_{i+1}]$. 
\end{definition}

\medskip\begin{proposition}\label{propPoly}Let $A=\{A_i\}_{1\leqslant i \leqslant n+1}$ be a polygon as  in Definition \ref{defPolygon}. Then  $V_A$ has bounded variation and 
\begin{eqnarray}
\delta V_A = \sum_{i=1}^n \left(\frac{A_{i}-A_{i-1}}{|A_i-A_{i-1}|}+ \frac{A_{i}-A_{i+1}}{|A_i-A_{i+1}|}\right)\delta_{A_i}. \label{varPolygone}
\end{eqnarray}
Here, $\delta_{A_k}$ is the Dirac measure at point $A_k$, and the vector in factor is pointing in the exact middle direction of the angle formed by the two sides of the polygon, with norm equal to $\sqrt{2(1+\cos(\theta_i))}$ where $\theta_i$ is the angle between the two vectors $\frac{A_{i}-A_{i-1}}{|A_i-A_{i-1}|}$ and $\frac{A_{i}-A_{i+1}}{|A_i-A_{i+1}|}$.
\end{proposition}

\begin{proof} Let $ \nu_i := (A_{i+1}-A_i)\vert A_{i+1}-A_i\vert^{-1}$ and $P_i=\Vect\{ A_{i+1}-A_i\}$. Then for $\theta\in C^1_c(\R^N,\R^N)$,  parameterizing the segment $[A_i,A_{i+1}]$ by $T_i(t) := (1-t)A_i + tA_{i+1}$, we see that for any $t\in(0,1)$
\begin{align*}
{\rm div}_{P_i}(\theta)(T_i(t))&=(D_{\nu_i}\theta\cdot\nu_i)(T_i(t))=\sum_{k=1}^N(D_{\nu_i}(\theta^k\partial_k)\cdot\nu_i)(T_i(t))\\
&=\sum_{k=1}^N(\nu_i(\theta^k)\nu_i^k)(T_i(t))=\frac{1}{|A_{i+1}-A_i|}\sum_{k=1}^N\frac{d}{dt}(\theta^k\circ T_i)(t)\nu_i^k
\end{align*}
where $(\partial_k)_{1\leqslant k\leqslant N}$ denotes the canonical basis of $\R^N$ and $(\theta^k)_{1\leqslant k\leqslant N}$ and $(\nu_i^k)_{1\leqslant k\leqslant N}$ are respectively the coordinates of $\theta$ and $\nu_i$ in this basis. Now we see that the first variation of $V_A$ is given by
\begin{align*}
\delta V_A(\theta)&:=\sum_{i=1}^n\int_{0}^1{\rm div}_{P_i}(\theta)(T_i(t))|T_i'(t)|dt\\
&=\sum_{i=1}^n\sum_{k=1}^N\int_{0}^1\frac{d}{dt}(\theta^k\circ T_i)(t)\nu_i^kdt\\
&=\sum_{i=1}^n(\nu_i\cdot(\theta(A_{i+1}) - \theta(A_i)))\\
&=\sum_{i=1}^n\theta(A_i)\cdot(\nu_{i-1}- \nu_i)
\end{align*}
with $\nu_n=\nu_0$ which proves \eqref{varPolygone}.
Let us now compute the norm of $\nu_{i-1}-\nu_i$ which is of the form $u+v$ with $|u|=|v|=1$. Then
$|u+v|^2=|u|^2+|v|^2-2\langle u,v\rangle=2+2\langle u,v\rangle=2(1+\cos(\theta)),$
where $\theta$ is the oriented angle between the two vectors $u,v$. 
\end{proof}

\subsection{General identities on varifolds} We state a integration by parts formula and adapt the Hsiung-Minkowski equality.

 \medskip\begin{proposition}[Integration by parts]\label{pIPP} Let $V$ be a $m$-varifold with locally bounded first variation. For every $u,v \in C^\infty(\mathbb{R}^N)$,
\begin{eqnarray}
 \int  \nabla_P u \cdot \nabla_P v \;dV = -\int  v \; {\rm div}_P(\nabla u) \; dV +\int_{\R^N}v(\nu\cdot\nabla u)d\|\delta V\| \label{IPP}
 \end{eqnarray} 
where for any $f\in C^\infty(\mathbb{R}^N)$, $\nabla_P f=\pi_P(\nabla f)$ for any $P\in G(N,m)$.  
 \end{proposition}
 
\begin{proof} Let $P\in G(N,m)$ and $(e_i)_{1\leqslant i\leqslant m}$ be an orthonormal basis of $P$, then
\begin{align}\label{derive}
{\rm div}_P(v\nabla u)&=\sum_{i=1}^m(D_{e_i}(v\nabla u))\cdot e_i\notag\\
&=\sum_{i=1}^m e_i(v)(\nabla u\cdot e_i)+v{\rm div}_P\nabla u\notag\\
&=\nabla_P u \cdot \nabla_P v+v{\rm div}_P\nabla u
\end{align}
On the other hand from (\ref{First-variation}) and (\ref{Riesz-representation})
\begin{equation}\label{HH}
\int{\rm div}_P(v\nabla u)(x)  \;dV(x,P)=\int_{\R^N}v(\nu\cdot\nabla u)d\|\delta V\|
\end{equation}
Therefore, integrating \eqref{derive} with respect to $V$ and using \eqref{HH}, yields \eqref{IPP} which finishes the proof of the proposition.
\end{proof} 

\medskip\begin{proposition}[Hsiung-Minkowski type equalities]  Let $V$ be an $m$-varifold with locally bounded first variation and $X$ be the identity vector field. Then 
\begin{eqnarray}
 \int |\nabla_P X|^2 dV =  \int {\rm div}_P(X) \;d V =m M,  \label{eq1}
  \end{eqnarray}
 where   $M$ is the total mass of the Varifold defined as
$$M:= V(\mathbb{R}^N\times G(N,m)).$$
\end{proposition}

\begin{proof} We first  write for any $P\in G(N,m)$ and orthonormal basis $(e_i)_{1\leqslant i\leqslant m}$ of $P$
\begin{align*} 
{\rm div}_P(X)&=\sum_{i=1}^m(D_{e_i}X)\cdot e_i=\sum_{i=1}^m\sum_{j=1}^N(D_{e_i}(X_j\partial_j))\cdot e_i\notag\\
&=\sum_{i=1}^m\sum_{j=1}^Ne_i(X_j)(\partial_j\cdot e_i)=\sum_{i=1}^m\sum_{j=1}^N (e_i\cdot\partial_j)^2=m
\end{align*}
where $(\partial_j)_{1\leqslant j\leqslant N}$ is the canonical basis of $\R^N$. Then
\begin{equation}\label{part1}
\int {\rm div}_P(X)(x) \;d V(x,P)=mM.
\end{equation} 
Moreover we observe that $\nabla X_i=\partial_i$ so that $D_{e_j}(\nabla X_i)=D_{e_j}\partial_i=0$ which implies 
$$\int  X_i \; {\rm div}_P(\nabla X_i) \; dV=0.$$
Therefore, if  we apply \eqref{IPP} with $u=v=X_i$ we get 
 \begin{eqnarray}
 \int  |\nabla_P X_i |^2dV &=&-\int  X_i{\rm div}_P(\nabla X_i)dV+\int_{\R^N}X_i(\nu\cdot\nabla X_i)d\|\delta V\| \notag \\
 &=&\int_{\R^N}X_i\nu_id\|\delta V\|,\notag 
 \end{eqnarray}
and summing over $i$ then using \eqref{First-variation}, \eqref{Riesz-representation} and \eqref{part1} yields,  
\begin{eqnarray}
\sum_{i=1}^N \int |\nabla_P X_i |^2  \;dV =\int_{\R^N}X\cdot\nu d\|\delta V\|= \int {\rm div}_P(X) \;d V=mM , \notag
 \end{eqnarray}
which proves the proposition.
\end{proof}


%
%
  
 \section{Reilly inequality for $H(2)$ varifolds} 
 
 \medskip\begin{definition}[Definition of $\lambda_1$]  Let  $V$ be an $m$-varifold in the class $H(2)$. We define the quantity
 \begin{eqnarray} \lambda_1:=\inf_{u \in C^\infty(\mathbb{R}^N) \setminus \{0\} \text{ s.t. } \displaystyle\int  u \; d\|V\|=0} \dfrac{\displaystyle\int |\nabla_P u(x)|^2 \;d V(x,P)}{\displaystyle \int_{\mathbb{R}^N} u^2\;d\|V\|}. \label{definitionL}\end{eqnarray}
 \end{definition}
 
If $V$ is the varifold associated to a smooth surface then $\lambda_1$ coincides with the usual first eigenvalue of the Laplacian. On the other hand, in the general context of $H(2)$ varifolds it is not clear  how to  rely  $\lambda_1$ with the first eigenvalue of an abstract unbounded operator. 

 \medskip\begin{proposition}[First order conditions in the class $H(2)$] \label{firstorder} Assume that $u \in C^\infty(\R^N)$ realizes the infimum in \eqref{definitionL}. Then for every $\varphi \in \Cc^{\infty}(\R^N)$ satisfying $\int_{\R^N} \varphi \; d||V||=0$ it holds
 \begin{eqnarray}
\int_{\R^N} \nabla_P u(x) \cdot \nabla_P \varphi(x) \;d V(x,P) = \lambda_1 \int_{\R^N } u \varphi \;d||V||. \label{cond1}
 \end{eqnarray}
 \end{proposition}
 
\begin{proof}  For $t\in \R$ we consider the admissible function    $u_t:=u+t\varphi$ as a test function, with $\varphi \in C^\infty_c(\R^N)$ satisfying $\int_{\R^N} \varphi \; d||V||=0$.   Since $u$ achieves the minimum of the Rayleigh quotient  we deduce that for all $t\in \R$, 
$$  \lambda_1=\frac{\int |\nabla_P u(x)|^2 \;d V_A(x,P)}{\int_{\mathbb{R}^N} u^2\;d||V||} \leqslant   \frac{\int |\nabla_P u_t(x)|^2 \;d V_A(x,P)}{\int_{\mathbb{R}^N} u_t^2\;d||V||}, $$
from which we easily obtain \eqref{cond1} by a standard variational argument. Indeed, let us introduce the real numbers
\begin{align*}
    D(u)&:=\int |\nabla_P u(x)|^2 \;d V_A(x,P),   &D(\varphi)&:=\int |\nabla_P \varphi(x)|^2 \;d V_A(x,P),\\ B(u,\varphi)&:=\int \nabla_P u(x) \cdot \nabla_P \varphi(x) \;d V_A(x,P), &E(u)&:= \int_{\mathbb{R}^N} u^2\;d||V||,\\ E(\varphi) &:=\int_{\mathbb{R}^N} \varphi^2\;d||V||, &C(u,\varphi)&:=\int u(x)\varphi(x) \;d||V||.
\end{align*}
Then for all $t\in \R$ we have
$$\lambda_1=\frac{D(u)}{E(u)}\leqslant \frac{D(u)+2tB(u,\varphi)+t^2 D(\varphi)}{E(u) +2 t C(u,\varphi) + t^2 E(\varphi)},$$
or differently,
$$\lambda_1(E(u) +2 t C(u,\varphi) + t^2 E(\varphi))\leqslant D(u)+2tB(u,\varphi)+t^2 D(\varphi).$$
Since $\lambda_1E(u)=D(u)$ we can actually simplify as follows 
$$\lambda_1(2 t C(u,\varphi) + t^2 E(\varphi))\leqslant 2tB(u,\varphi)+t^2 D(\varphi).$$
Dividing by $2t>0$ and letting $t\to 0$ we get 
$$\lambda_1 C(u,\varphi) \leqslant B(u,\varphi).$$
Similarly, dividing by $2t<0$ and letting $t\to 0$ we get the opposite 
$$\lambda_1 C(u,\varphi) \geqslant B(u,\varphi),$$
thus in conclusion we have obtained 
$$\lambda_1 C(u,\varphi)  =B(u,\varphi),$$
and this achieves the proof.
\end{proof}


\begin{proof}[Proof of Proposition \ref{reilly1}]
Since  $\|V\|$ has finite mass,    up to use a translation we can assume that
$$ \int_{\mathbb{R}^N}X \; d \|V\|=0.$$
 Then we take $X_i$ as a test function in the definition of $\lambda_1$ which yields,
 \begin{eqnarray}
 \lambda_1 \int_{\mathbb{R}^N} (X_i)^2 \;d \|V\| \leqslant \int |\nabla_P X_i|^2 \; dV(X,P) \label{lambda1X}
 \end{eqnarray}
 which after summation becomes
 \begin{eqnarray}
 \lambda_1 \int_{\mathbb{R}^N} |X|^2  \;d \|V\| \leqslant \sum_{i=1}^{N}\int  |\nabla_P X_i|^2 \; dV(X,P). \label{ineq11}
 \end{eqnarray}
 Then we can use Cauchy-Schwarz (for vectors of $\mathbb{R}^N$), H\"older and  Hsiung-Minkowski \eqref{eq1} to obtain
\begin{eqnarray}
 \lambda_1 \int_{\mathbb{R}^N} |X|^2  \;d \|V\| &\leqslant& \sum_{i=1}^{N}\int |\nabla_P X_i|^2 \; dV(x,P)=  \int_{\mathbb{R}^N} {\rm div}_P(X)dV \label{ligne1}\\
 &=& \frac{\left( \int_{\mathbb{R}^N} {\rm div}_P(X)dV\right)^2}{mM} = \frac{\left(\int_{\mathbb{R}^N} X \cdot {\bf H}\;d\| V\| \right)^2}{mM}  \notag \\
 &\leqslant &  \frac{ \left( \int_{\mathbb{R}^N} |X | |\Hm | \;d\|V\|  \right)^2}{mM} \leqslant   \frac{\|X\|_{L^2}^2 \|{\bf H}\|_{L^2}^2}{mM} \notag
 \end{eqnarray}
 which proves \eqref{ineqR}.
  \end{proof}


\subsection{Equality case in the class $H(2)$}

%

\

\

\begin{proof}[Proof of Theorem \ref{equalityH2}] If there is equality in  \eqref{ineqR}, then all the inequalities that have been used in the proof of the Reilly inequality \eqref{ineqR}, are equalities. In particular the equality in Cauchy-Schwarz inequality \eqref{ineq11} says that there exists a constant $\alpha \in \R$ such that
 \begin{eqnarray}
  H=\alpha X, \quad  ||V||-a.e. \label{colinear1}
 \end{eqnarray} 
 This implies that $\pi_P(X)=0$  a.e  thanks to Proposition \ref{orthogonality}. Let us now prove that  $\alpha=\lambda_1$. Indeed, we know that the inequality in  \eqref{lambda1X} is an equality, thus for all $1\leqslant k \leqslant N$, the coordinate function $X_k$ realizes the infimum in the Rayleigh quotient. Then using the optimality condition \eqref{cond0},  it follows that  for every $\varphi \in C^{\infty}(\R^N)$  satisfying $\int_{\R^N} \varphi \; d||V||=0$ and for every $1\leqslant k \leqslant N$ it holds
 \begin{eqnarray}
\int_{\R^N} \nabla_P X_k(x) \cdot \nabla_P \varphi(x) \;d V(x,P) = \lambda_1 \int_{\R^N } X_k \varphi \;d||V||. \label{cond0}
 \end{eqnarray}
 Then we integrate by parts using  \eqref{IPP} yielding
 \begin{eqnarray}
 \int_{\R^N} \nabla_P X_k(x) \cdot \nabla_P \varphi(x) \;d V(x,P) &=& -\int  \varphi {\rm div}_P(\nabla X_k) \; dV + \int_{\mathbb{R}^N} \varphi   \nabla X_k \cdot \Hm \; d\|  V\| \notag \\
 &=&  \int_{\mathbb{R}^N} \varphi\,  \Hm_k \; d\|  V\| ,
 \end{eqnarray}
because $\nabla X_k=\partial_k$ and ${\rm div}_P(\nabla X_k)=0$. We deduce that for all $\varphi \in C^{\infty}(\R^N)$  satisfying $\int_{\R^N} \varphi \; d||V||=0$ and for every $1\leqslant k \leqslant N$ it holds
$$\int_{\mathbb{R}^N} \varphi\,  \Hm_k d\|  V\| =\lambda_1 \int_{\R^N } X_k \varphi \;d||V||,$$
 and since $\Hm_k=\alpha X_k$,  we must actually have $\alpha=\lambda_1$. Our next goal is to prove that $|X|$ is constant on the support of $||V||$, which would prove that the varifold is supported on a sphere.  To see this, we compute in two different ways the quantity $ \int  \nabla_P |X|^2 \cdot \nabla_P \varphi \;dV$ for  $\varphi \in C^\infty(\R^N)$, a test function. Using first that $\nabla |X|^2=2X$  and that $\pi_P(X)=0$ we get
$$ \int  \nabla_P |X|^2 \cdot \nabla_P \varphi \;dV=\int 2(\pi_PX )\cdot \nabla_P \varphi \;dV=0.$$
Then using the integration by parts formula \eqref{IPP} with the $C^\infty$ functions $|X|^2$ and $\varphi$ we obtain
 \begin{eqnarray}
 0= \int  \nabla_P |X|^2 \cdot \nabla_P \varphi \;dV &=& -\int  \varphi \; {\rm div}_P(\nabla |X|^2) \; dV + \int_{\mathbb{R}^N} \varphi\,  \nabla |X|^2 \cdot \Hm d\|  V\| \notag \\
 &=& -\int  \varphi \; {\rm div}_P(2X) \; dV + \int_{\mathbb{R}^N} \varphi\, 2X \cdot \Hm d\|  V\| \notag \\
 &=& -\int_{\R^N}  \varphi \; 2m \; d\|V\| + 2\lambda_1 \int_{\mathbb{R}^N} \varphi  |X|^2\,  d\|  V\|,
 \end{eqnarray} 
where we have used that $\Hm=\lambda_1 X$. In conclusion we have obtained, 
$$\forall \varphi\in C^{\infty}(\R^N), \quad 2m\int_{\R^N}  \varphi  \; d\|V\| = 2\lambda_1 \int_{\mathbb{R}^N} \varphi  |X|^2\,  d\|  V\|.$$
Since $\|V\|$ is not trivial, by taking a particular $\varphi$ it is easy to see  that $\lambda_1 \not =0$ and the above identity says that $|X|$ is constant and 
$$|X|=\sqrt{\frac{m}{\lambda_1}} \quad \|V\|-a.e.$$ In other words denoting $R:=\sqrt{m/\lambda_1}$ we   have obtained that ${\rm supp}(\|V\|) \subset \partial B(0,R)$ which proves the first half of the statement. Let $L={\rm supp}(\|V\|)$ and consider $V^S=g\HH^m\mid_{L}\otimes\delta_{TL}$. Let $\theta\in\mathcal{X}(S_R)$. Then $\theta$ can be extended in a smooth field $\tilde{\theta}$ of $\R^N$ and at $x\in S_R$, we have for any $Y\in T_xS_R$ :
\begin{align*}
D^{S_R}_Y\theta&=\pi_{T_xS_R}(D_Y\tilde{\theta})=D_Y\tilde{\theta}-\frac{1}{R^2}((D_Y\tilde{\theta})\cdot x)x\\
&=D_Y\tilde{\theta}+\frac{1}{R^2}(\theta\cdot D_Yx) =D_Y\tilde{\theta}+\frac{1}{R^2}(\theta\cdot Y)x.
\end{align*}
Then for a.e. $x\in L$ and orthonormal basis $(e_i)_{1\leqslant i\leqslant m}$ of $T_xL$, we have 
\begin{align*}
({\rm div}^{S_R}_{TL}\theta)(x)&=\sum_{i=1}^m(D^{S_R}_{e_i}\theta\cdot e_i)\\
&=\sum_{i=1}^m(D_{e_i}\tilde{\theta}\cdot e_i)+\frac{1}{R^2}\sum_{i=1}^m(e_i\cdot\tilde{\theta})(x\cdot e_i)\\
&=({\rm div}_{TL}\tilde{\theta})(x).
\end{align*} 
It follows that
\begin{align*}
\int{\rm div}^{S_R}_P(\theta)(x)dV^S(x,P)&=\int_{L}g(x)({\rm div}^{S_R}_{T_xL}\theta)(x) \;d\mathcal{H}^m(x)\\
&=\int_{K}g(x)({\rm div}_{T_xK}\tilde{\theta})(x) \;d\mathcal{H}^m(x)\\
&=\int{\rm div}_P(\tilde{\theta})dV(\cdot,P)\\
&=-\int\theta\cdot\Hm d\|V\|=0.
\end{align*}
Then denoting by $\Hm^S$ the generalized mean curvature vector of $V^S$, we have for all $\theta\in\mathcal{X}(S_R)$ 
$$\int{\rm div}^{S_R}_P(\theta)(x)dV^S(x,P)=-\int_{S_R}\theta\cdot\Hm^Sd\|V^S\|+\int_{S_R}\theta\cdot\eta d\|\delta V^S\|_{sing}=0$$
where $\eta$ is a $\|\delta V^S\|$ measurable vector field on $S_R$. Since $d\|V^S\|$ and $d(\delta V^S)_{sing}$ are mutually singular, they are both zero. Consequently, $V^S$ is $H(2)$ and $\Hm^S=0$. 

To finish the proof in the case $m=N-1$, we show  that the support of the varifold is the whole sphere $S_R$ and the multiplicity $g$ of  $V$ is constant on the whole sphere. Then for any smooth vector field $\Phi$ with compact support in $\R^N$, we have :
$$\int_K g{\rm div}_{T_xS_R}(\Phi)d\mathcal H^{N-1}=\int_L g\Hm\cdot\Phi d\mathcal H^{N-1}=\frac{m}{R^2}\int_{S_R}{\bf 1}_L g X\cdot\Phi d\mathcal H^{N-1}$$
Moreover if we consider an orthonormal basis $(e_i)_{1\leqslant i\leqslant N-1}$ on $T_xS_R$ :
\begin{align*}
{\rm div}_{T_xS_R}(\Phi)&={\rm div}_{T_xS_R}(\Phi^T)+\frac{1}{R^2}{\rm div}_{T_xS_R}((\Phi\cdot X)X)\\
&={\rm div}_{T_xS_R}(\Phi^T)+\frac{1}{R^2}\sum_{i=1}^{N-1}D_{e_i}((\Phi\cdot X)X)\cdot e_i\\
&={\rm div}_{T_xS_R}(\Phi^T)+\frac{1}{R^2}\sum_{i=1}^{N-1}e_i(\Phi\cdot X)(X\cdot e_i)+\frac{m}{R^2}(\Phi\cdot X)\\
&={\rm div}_{T_xS_R}(\Phi^T)+\frac{m}{R^2}(X\cdot\nabla^S(\Phi\cdot X))+\frac{m}{R^2}(\Phi\cdot X)
\end{align*}
where $\nabla^S(\Phi\cdot X)$ is the gradient on $S_R$. Then $(X\cdot\nabla^S(\Phi\cdot X))=0$ and we deduce that :
$$\int_{S_R}{\bf 1}_L g{\rm div}_{T_xS_R}(\Phi^T)d\mathcal H^{N-1}=0$$
Since it is true for any smooth vector field $\Phi$ with compact support in $\R^N$ (see Proposition~\ref{f-constant} in the appendix), it follows that ${\bf 1}_L g$ is constant and $L=B(0,R)$ and $g$ is constant.
\end{proof}

\begin{remark} It is easy to check that  the varifold   associated to $n\mathcal{H}^{N-1}|_{S_R}$ with $n \in \N^*$ satisfies the equality case in Reilly's inequality. To see this we know that, denoting $m=N-1$, the mean curvature  $\bf H$ of the sphere of radius $R$ is equal to $ \nu m/R$, where $\nu$ is the outwise normal vector, in other words
$$\int_{S_R}  {\rm div}_P(\nabla \varphi) \;d\mathcal{H}^{N-1} = \int_{S_R} \frac{m}{R} \varphi \cdot \nu \; d\mathcal H^{N-1} .$$

Multiplying by $n$ on each sides we see that  the the mean curvature $\bf H$ associated to the varifold $n\mathcal{H}^{N-1}|_{S_R}$ of multiplicity $n$ is still $ \nu m/R$. Using $R=\sqrt{m/\lambda_1}$ we deduce that 
$$\|{\bf H}\|_{L^2(d\|V\|)}^2 = \int_{S_R} \frac{m^2}{R^2} nd \mathcal{H}^{N-1}=\frac{nm^2}{R^2} \mathcal{H}^{N-1}(S_R)=\lambda_1 mM,$$
because 
$$M=\|V\|(S_R)=n \mathcal{H}^{N-1}(S_R).$$
\end{remark}
   
 \section{Reilly inequality for polygons}
\label{polygonSection}
 \subsection{The space $H^1(K_A)$}
 
 In this section  we define a Sobolev space $H^1(K_A)$ when $K_A$ is a polygon. For simplicity we will  sometimes denote by $K$ the polygon, instead of $K_A$.  In what follows we denote by $L^2(K)$ the usual complete space  $L^2(K,d\mathcal{H}^1)$ containing measurable functions $u$ such that $\int_{K} u^2 \, d\mathcal{H}^1<+\infty$, and defined $\mathcal{H}^1$-a.e.\! on $K$. Following a general approach introduced in \cite{zombie}, we describe below a how to define a Dirichlet energy associated to the 1-rectifiable set $K$. Since $K$ is $1$-rectifiable we know the existence of  an approximative tangent line at $\mathcal{H}^1$-a.e.\! point $x\in K$, and  for any such point $x$ we can choose a unit tangent vector $\tau_K(x)$ in direction of that line. For every smooth function $u \in C^\infty(\R^N)$ we introduce 
$$N(u)=\left(\int_{K} u^2(y) \; d\mathcal{H}^1(y) + \int_{K} |\nabla u \cdot \tau_K(y)|^2 \; d\mathcal{H}^1(y)\right)^\frac{1}{2}.$$
Notice that $N(u)$ involves only the trace on $K$ of the smooth function $u$ defined on the whole $\R^N$. Then we consider the space $\mathcal{D}(K)$ as being the restriction on $K$ of  $C^\infty(\R^N)$ functions, and we endow this space with the norm  $N$. In particular,  $\mathcal{D}(K)$  is a subspace of $L^2(K)$.  Finally, we define $H^1(K)$ as follows. 

\medskip\begin{definition}[Space $H^1(K)$ and $\nabla_K u$] For a polygon $K$, we define $H^1(K)$ as the completion of $\mathcal{D}(K)$ for the norm $N$. In particular, $H^1(K)$ is a closed subspace of $L^2(K)$ for which $\mathcal{D}(K)$ is a dense subset. For any $u \in H^1(K)$ we define $\nabla_K u \in L^2(K;\R^N)$ as the $L^2$~limit of the projection (a.e.) of $\nabla u_n$ on $K$, that is, $(\nabla u_n\cdot \tau_K)\tau_K$, for $u_n \to u$ in the norm $N$. In particular $\nabla_K u$ does not depend on the choice of the sequence $u_n \in \mathcal{D}(K)$  such that $u_n \to u$ in $H^1(K)$, in the equivalent class of Cauchy sequences for the norm $N$, which justifies the definition.
\end{definition}

The construction of $H^1(K)$ is rather standard. A way to define it rigorously  is for instance by considering all Cauchy sequences for $N$ in $\mathcal{D}(K)$, on which one defines the following equivalence relation: two Cauchy sequences $u_n$ and $v_n$ are equivalent if and only if $N(u_n-v_n)\to 0$. Then $H^1(K)$ is the quotient of all Cauchy sequences by this relation. It is easy to see  that this space is a complete space for which $\mathcal{D}(K)$ is a dense subset. 

  In particular, we can consider  $u \in H^1(K)$ as being a function $u\in L^2(K)$ for which there exists a sequence $u_n \in C^\infty(\R^N)$ such that $u_n\vert_K \to u$ in $L^2$ and   $(\nabla u_n \cdot \tau_K)\tau_K$ has a limit in $L^2(K;\R^N)$. For $u \in H^1(K)$ we will denote by $\nabla_K u$ the  limit  of $(\nabla u_n \cdot \tau_K)\tau_K$. By construction,  the limit of $(\nabla u_n \cdot \tau_K)\tau_K$ does not depend on the choice of the sequence $u_n$, chosen in the equivalent class of Cauchy sequences, and $\nabla_K u$ is therefore well defined. Next, our Dirichlet energy, defined on $H^1(K)$, is given by
$$\int_{K} |\nabla_K u|^2 \;d\mathcal{H}^1.$$
Of course if $u\in  C^\infty(\R^N)$, $\nabla_K u = (\nabla u \cdot \tau_K)\tau_K$ $\mathcal{H}^1$-a.e., thus the Dirichlet energy coincides with the natural one in that case.


Now we would like to establish an analogue of Sobolev embedding and Rellich theorem within our context.

\medskip\begin{proposition}\label{embb} If $K$ is a polygon, then for all $u\in H^1(K)$,  
\begin{eqnarray}
|u(x)-u(y)|\leqslant {\dist}_K(x,y)^\frac{1}{2} \|\nabla_K u\|_{L^2(K)}, \quad \quad \text{ for } \mathcal{H}^1-\text{a.e. } x,y \in K, \label{inequS} 
\end{eqnarray}
where ${\dist}_K(x,y)$ is the geodesic distance on $K$.  
\end{proposition}

\begin{proof} Assume first that $u \in C^\infty(\R^N)$ and let $x,y \in K$ be given. We know that there exists a geodesic Lipschitz curve $\gamma : [0,L]\to K$ with $L={\dist}_K(x,y)$,  which is injective,   parametrized with constant speed so that 
$|\gamma'(t)|=1$, such that $\gamma(0)=x$, $\gamma(L)=y$ and 
$$\int_{0}^L |\gamma'(s)| \;ds= {\rm dist}_K(x,y).$$
The function $u\circ \gamma :[0,L]\to \R$ is Lipschitz continuous,  thus in particular absolutely continuous and therefore
$$u(x)-u(y)=\int_0^L \nabla u(\gamma (t))\cdot\gamma'(t) \;dt,$$
from which we easily deduce that 
\begin{eqnarray}
|u(x)-u(y)| &\leqslant& \left(\int_0^L |\gamma'(t)|^2\, dt\right)^{\frac{1}{2}} \left(\int_0^L (\nabla u(\gamma (t)) \cdot \gamma'(t) )^2 \;dt \right)^{\frac{1}{2}} \notag \\
&\leqslant &{\rm dist}_K(x,y)^\frac{1}{2} \|\nabla_K u\|_{L^2(K)} , \notag
\end{eqnarray}
which  proves \eqref{inequS} in the case of a smooth $u\in \mathcal{D}(K)$. Now if $u\in H^1(K)$ we know by definition that there exists a sequence of functions $u_n \in \mathcal{D}(K)$ such that $u_n\to u$ in $L^2(K)$ and $\nabla_K u_n \to \nabla_K u$   in $L^2(K)$. Up to extracting a subsequence we can assume that $u_n\to u$ $\mathcal{H}^1$-a.e.\! on $K$. Applying \eqref{inequS}   to $u_n$ and then passing to the limit we  then conclude that \eqref{inequS} also holds for $u$.
\end{proof}

From Proposition \ref{embb} we get the following immediate corollary.

\medskip\begin{corollary} If $K$ is a polygon,  then every function $u\in H^1(K)$ admits an $L^2$-representative which is continuous.
\end{corollary}

We will also need the following $L^\infty$ estimate.

\medskip\begin{corollary}\label{bounded} If $K$ is a polygon,  then every function $u\in H^1(K)$ is bounded. Moreover,
$$\|u\|_{L^\infty(K)}\leqslant  \frac{1}{\mathcal{H}^1(K)^\frac{1}{2}}\|u\|_2+  (\mathcal{H}^1(K))^{\frac{1}{2}}\|\nabla_K u\|_2.$$
\end{corollary}

\begin{proof} For $u \in H^1(K)$ we know from Proposition \ref{embb} that for $\mathcal{H}^1$-a.e.\!  $x,y \in K$,
$$|u(x)-u(y)| \leqslant \|\nabla_K u\|_2{\rm dist}_K(x,y)^{\frac{1}{2}} .$$
In particular,
$$u(x)-u(y) \leqslant \|\nabla_K u\|_2 (\mathcal{H}^1(K))^{\frac{1}{2}},$$
thus  integrating with respect to $y \in K$  and dividing by $\mathcal{H}^1(K)$ we get 
$$u(x)-\frac{1}{\mathcal{H}^1(K)}\int_K u(y)  \;d\mathcal{H}^1(y)   \leqslant \|\nabla_K u\|_2 (\mathcal{H}^1(K))^{\frac{1}{2}}.$$
Finally using H\"older inequality,
$$u(x)  \leqslant  \|\nabla_K u\|_2 (\mathcal{H}^1(K))^{\frac{1}{2}} + \frac{1}{\mathcal{H}^1(K)^\frac{1}{2}}\|u\|_2.$$
 Reasoning the same way with $-u$ we get 
$$|u(x)|  \leqslant  \|\nabla_K u\|_2 (\mathcal{H}^1(K))^{\frac{1}{2}} + \frac{1}{\mathcal{H}^1(K)^\frac{1}{2}}\|u\|_2,$$
which proves the Corollary.
\end{proof}
 
We can also prove the following compact embedding result.

\medskip\begin{corollary} \label{compact}Let $K$ be a polygon.  The embedding $H^1(K)\hookrightarrow L^2(K)$ is compact. More precisely, from every bounded sequence $(u_n)_{n\in \mathbb{N}}$ in $H^1(K)$ we can extract a uniformly converging sequence, and in particular a converging sequence in $L^2(K)$.
\end{corollary}

\begin{proof} Let $(u_n)$ be a bounded sequence in $H^1(K)$. Then for each $n$ we consider the specific $L^2$ representative of $u_n$ for which  Proposition \ref{embb}  yields  the  estimate 
$$|u_n(x)-u_n(y)| \leqslant C{\rm dist}_K(x,y)^{\frac{1}{2}},$$
where the constant $C$ is uniform in $n$.  In particular the sequence $(u_n)$ is equicontinuous on the compact set $K\subset \R^2$.  Moreover applying Corollary \ref{bounded} we know that  $(u_n)$ is also equibounded. Thank to Arzel\`a-Ascoli theorem, we deduce the existence of a subsequence $(u_{n_k}) $ that converges uniformly on $K$. This achieves the proof of the Corollary.
\end{proof}

 \subsection{Definition of $\lambda_1$ and spectral theory}

 \medskip\begin{definition}[Definition of $\lambda_1^A$] \label{deflambda1} Let  $V_A$ be the rectifiable $1$-varifold associated to a polygon $A=\{A_i\}_{1\leqslant i \leqslant n+1}$ as defined in Section \ref{polygon}. Then we define the discreate measure 
 $$\mu_A:=\sum_{i=1}^n\delta_{A_i},$$
 and the vectorial function $${\bf H}(x)= \left(\frac{A_{i}-A_{i-1}}{|A_i-A_{i-1}|}+ \frac{A_{i}-A_{i+1}}{|A_i-A_{i+1}|}\right) {\bf 1}_{\{A_i\}}(x)$$
 in such a way that $\delta V_A= {\bf H}\mu_A$.
 Then we finally define the  quantity
 \begin{eqnarray}
 \lambda_1^A:=\inf_{u \in C^\infty(\mathbb{R}^N) \setminus \{0\} \text{ s.t. } \int_{\R^N} u \; d \mu_A=0} \frac{\int |\nabla_P u(x)|^2 \;d V_A(x,P)}{\int_{\mathbb{R}^N} u^2\;d\mu_A}. \label{lambdaP}
 \end{eqnarray}
 \end{definition}
 
 As we will see later, this definition of $\lambda_1^A$ is the right one in order to obtain a Reilly inequality. Before that, we would like to link $\lambda_1^A$ with the first eigenvalue of a certain operator.  We could try to apply the standard spectral theory for quadratic forms directly with the space $H^1(K)$, but  the problem comes from the fact that the Dirichlet energy is not coercive with respect to the norm $\int_{\mathbb{R}^N} u^2\;d\mu_A$.  Therefore, we will first reduce to a finite dimensional space via the following proposition, that says that an eigenfunction is the solution of the following Robin type eigenvalue problem
 $$
 \left\{
 \begin{array}{ll}
  u'' = 0,& \text{ on } ]A_i , A_{i+1}[,\\
\displaystyle \frac{\partial u}{\partial \nu^+}  -  \frac{\partial u}{\partial \nu^-} =\lambda_1 u, & \text{ on } A_i.
 \end{array}
\right.
 $$
This shows that $u$ must be piecewise affine and will allow us to reduce the study to a finite dimensional space.
 
 \medskip\begin{proposition}[First order conditions] \label{firstorder1}  The infimum in \eqref{lambdaP} is achieved  in the space $H^1(K_A)$ for a function $u$ that satisfies the following first order condition: let  $T_i$ be the parameterization of $[A_i,A_{i+1}]$ defined by $T_i(t) := (1-t)A_i + tA_{i+1}$. We also denote by $\ell_i:=|A_{i+1}-A_i|$. Then there are  some constants $\alpha_i,\beta_i \in \R$ such that
$$
\left\{
 \begin{array}{ll}
u\circ T_i(t)=\alpha_i +t \beta_i, \quad  &\textnormal{ for } t\in[0,1] \textnormal{ and }  1\leqslant i \leqslant n,\\
\alpha_i+\beta_i=\alpha_{i+1},\\\displaystyle 
\lambda_1\alpha_i=\left(\frac{\beta_{i-1}}{\ell_{i-1}} - \frac{\beta_i}{\ell_i}\right), \quad &\textnormal{for } 1\leqslant i  \leqslant n.
 \end{array}
  \right. 
 $$
and $\alpha_{n+1}=\alpha_1$ and $\beta_{n+1}=\beta_1$.
 \end{proposition}
 
\begin{proof}  The existence of a minimum $u$ in $H^1(K_A)$ follows from the direct method of calculus of variations due to the compact embedding of Corollary \ref{compact}. The details are left to the reader. Now for  $t\in \R$ we consider a function of the form  $u_t:=u+t \theta$ as a test function, with $\theta \in C^\infty_c(\R^N)$ satisfying ${\rm supp}(\theta) \cap \{A_i\}=\emptyset$. Then $\int_{\R^N}(u+t \theta )\mu_A=0$ and $\int_{\mathbb{R}^N} u_t^2\;d\mu_A=\int_{\mathbb{R}^N} u^2\;d\mu_A$ for all $t \in \R$. Since $u$ is an eigenfunction we deduce that 
 $$ \forall t \in \mathbb{R},\quad \int_{K_A} |\nabla_P u(x)|^2 \;d  \mathcal{H}^1  \leqslant  \int_{K_A} |\nabla_P u_t(x)|^2 \;d \mathcal{H}^1. $$
 In other words the restriction of $u$ to each segment $[A_i,A_{i+1}]$ must be harmonic, and therefore an affine function. We deduce that if $T_i$ is the parameterization of $[A_i,A_{i+1}]$ as in the statement  of the proposition, then there exists some constants $\alpha_i,\beta_i \in \R$ such that 
$u\circ T_i(t)=\alpha_i + t \beta_i$. We also get the condition $\alpha_i+\beta_i=\alpha_{i+1}$ because $u$ must be globally continuous on $K_A$ (because $u$ belongs to $H^1(K_A)$ and $K_A$ is one dimensional). 

 Now we consider again a function of the form  $u_t:=u+t ( \theta-1)$ as a test function, but now with $\theta$ whose support touching one vertex $A_{i_0}$, with $\theta(A_{i_0})=n$. By this way we have $\int_{\R^N} \theta-1 \; d\mu_A=0$ and $u_t$ is an admissible test function for the Rayleigh quotient. Since $u$ is an eigenfunction we must have 
  
$$ \lambda_1=\frac{\int |\nabla_P u(x)|^2 \;d V_A(x,P)}{\int_{\mathbb{R}^N} u^2\;d\mu_A}\leqslant   \frac{\int |\nabla_P u_t(x)|^2 \;d V_A(x,P)}{\int_{\mathbb{R}^N} u_t^2\;d\mu_A}.$$
By reasoning exactly as in the proof of Proposition \ref{firstorder} we deduce that 
$$\lambda_1 \sum_{i=1}^n  u(A_i)( \theta(A_i)-1) = \int_{K_A} \nabla_P u \cdot \nabla_P \theta \; d\mathcal{H}^1.$$
But since 
$$ \sum_{i=1}^n  u(A_i)=0 \quad \text{ and } \quad \theta(A_j)=0 \text{ for }j \not = i_0,$$
we finally obtain
$$n\lambda_1   u(A_{i_0})   = \int_{K_A} \nabla_P u \cdot \nabla_P \theta \; d\mathcal{H}^1.$$
Then we notice that 
$$u(A_{i_0})=\alpha_{i_0}.$$
We denote by $\ell_i:=|A_{i+1}-A_i|$ in such a way that $T_i$ is a parametrization of the segment $[A_i,A_{i+1}]$ with constant speed $\ell_i$. In other words, any smooth function $g$ on the segment $[A_i,A_{i+1}]$ can be integrated via the formula
$$\int_{[A_i,A_{i+1}]} g \;d\mathcal{H}^1 = \ell_i\int_{0}^1 g\circ T_i   \; dt.$$
Moreover, 
$$(g\circ T_i(t))'=\ell_i(\nabla_Pg)(T_i(t)),$$
so that 
$$\int_{[A_i,A_{i+1}]} \nabla_P u \cdot \nabla_P \theta \; d\mathcal{H}^1 =\frac{1}{\ell_i}\int_0^1 (u\circ T_i)'(\theta\circ T_i)' \; dt $$
Then integrating on $K_A$,
\begin{eqnarray}
\int_{K_A} \nabla_P u \cdot \nabla_P \theta \; d\mathcal{H}^1 &=& \sum_{i=1}^n \frac{1}{\ell_i} \int_{0}^1 \beta_i  (\theta\circ T_i)'(t) dt  \notag \\
&=&\theta(A_{i_0})\left(\frac{\beta_{i_0-1}}{\ell_{i_0-1}} - \frac{\beta_{i_0}}{\ell_{i_0}}\right)    \notag 
\end{eqnarray}
which achieves the proof.
 \end{proof}

  The above proposition  shows that all eigenfunctions are piecewise affine functions that are moreover continuous accros the vertices (because there are in the space $H^1(K)$). We can identify this space with the following one:
  $$E:=\left\{(\alpha,\beta) \in \R^n \times \R^n  \;|\; \text{ such that }  \alpha_i+\beta_i =\alpha_{i+1} ,  \quad 1\leqslant i \leqslant n\right\},$$
where a piecewise affine function $u$ is taken of the form $\alpha_i +\beta_i t$ on the edge $[A_i,A_{i+1}]$, of length $\ell_i$, parameterized by $T_i(t) := (1-t)A_i + tA_{i+1}$. Then we can identify our $\lambda_1$ as being the first non zero eigenvalue of a finite dimensional space, as  summarized in the following statement that follows from standard (finite dimensional) spectral theory.
 
\medskip\begin{corollary} \label{spectrumfinite} Let  $A=\{A_i\}_{1\leqslant n+1}$  be a polygon with $n$ vertices  (i.e. $A_{n+1}=A_1$) and let $\ell_i=|A_{i+1}-A_i|$ be the length of its sides. Let  $E\subset \R^n \times \R^n$ be the finite dimensional vectorial space defined by
$$E:=\left\{(\alpha,\beta) \in \R^n \times \R^n  \;|\; \text{ such that }  \alpha_i+\beta_i =\alpha_{i+1} ,  \quad 1\leqslant i \leqslant n\right\},$$
where by convention we have denoted  in the above, $\alpha_{n+1}=\alpha_1$ and $\beta_{n+1}:=\beta_1$. Then  $E$  endowed with the norm  
\begin{eqnarray}
\|(\alpha,\beta)\|_E :=\left(\sum_{i=1}^n \alpha_i^2\right)^{1/2} \label{essS}
\end{eqnarray}
is an Euclidean space and the following quadratic form
$$Q(\alpha,\beta):=\sum_{i=1}^n \frac{\beta_i^2}{\ell_i}$$
defines an associated  symmetric linear operator $M:E\to E$ such that 
$$\forall  u \in E, \quad Q(u)=\langle M u , u \rangle.$$
In particular,  $M$ admits  a finite number of eigenvalues $(\lambda_k)_{0\leqslant k \leqslant n-1}$ such that  $\lambda_0=0$ and $\lambda_1=\lambda_1^A$.  Moreover,
$$\lambda_1^A=\lambda_1(M)=\min_{u \in E, u\not = 0, \sum_{i=1}^n\alpha_i=0} \frac{Q(u)}{\|u\|_E^2}.$$
\end{corollary}

\begin{proof} The Corollary follows rather directly from standard finite dimensional spectral analysis. The main point is that, restricted to the subspace $E$, the quadratic form \eqref{essS} becomes a norm. Indeed, if $\sum_{i=1}^n \alpha_i^2=0$ then all the $\beta_i$ are zero because the definition of $E$ implies $\beta_i=\alpha_{i+1}-\alpha_i$.

  Then,  identifying any piecewise affine function equal to $\alpha_i+t \beta_i$ on $[A_i,A_{i+1}]$ with the corresponding element in the space $E$, the definition of  $Q$ has been taken in such a way that 
$$Q(u)= \sum_{i=1}^n \frac{\beta_i^2 }{\ell_i} = \int_{K_A} |\nabla_P u|^2 \;d  \mathcal{H}^1.$$
In other words the Rayleigh quotient on functions, is nothing but the following Rayleigh quotient in $E$,
$$ Q(u)/\|u\|_E^2,$$
and by definition of $\lambda_1^A$ we have
$$\lambda_1^A=\min_{u \in E, u\not = 0, \sum_{i=1}^n\alpha_i=0} \frac{Q(u)}{\|u\|_E^2}.$$
But now the above minimisation problem can be solved by    diagonalizing  the quadratic form $Q$. Let $B(u,v):=\frac{1}{4}(Q(u+v)-Q(u-v))$ be the symmetric bilinear form associated with $Q(u)$, and let $\langle  u,v\rangle_E$ be the scalar product associated with the norm $\|u\|_E$. Then by Riesz theorem, for any $v\in E$ there exists $w \in E$ such that $B(u,v)=\langle w, v\rangle_E$, and the mapping $v\mapsto w$ is linear. Denoting by $M$ this mapping we have $B(u,v)=\langle M u,v \rangle$ and in particular
$$Q(u)=\langle M u,u \rangle \quad \quad \text{ for all } u \in E.$$
Moreover,  $M$ is symmetric because $\langle M u,v \rangle=B(u,v)=B(v,u)=\langle M v,u \rangle$. Therefore, $M$ admits $n$ real eigenvalues $\lambda_k$, associated to  an orthonormal basis for the norm $\|\cdot\|_E$. Notice that the kernel of  $Q$ is exactly the constant functions, or equivalently, the elements $(\alpha_i,\beta_i)$ of $E$ such that $\beta_i=0$ for all $i$ and $\alpha_i=\alpha_{i+1}$, denoted by $F$. Notice also that $F^\bot$ is exactly given by the elements of $E$ satisfying $\sum_{i=1}^n\alpha_i=0$. We deduce that $\lambda_0=0$ and that 
$$\lambda_1(M)=\min_{u \in E\cap F^\bot, u\not = 0} \frac{Q(u)}{\|u\|_E^2}=\min_{u \in E, u\not = 0, \sum_{i=1}^n\alpha_i=0} \frac{Q(u)}{\|u\|_E^2}=\lambda_1^A,$$
which finishes the proof of the Corollary. 
\end{proof}
  
  \medskip\begin{definition}[Spectrum of a polygon] We define the spectrum $(\lambda_k)_{0\leqslant k\leqslant n-1}$ of a polygon  as being the spectrum of the linear operator $M$ defined in Corollary \ref{spectrumfinite}. In particular, $\lambda_1$ coincides with   $\lambda_1^A$ as in Definition \ref{deflambda1}. 
    \end{definition}

   \medskip\begin{proposition}[First order condition for $\lambda_k$] \label{firstorder2} Under the same notation as in the statement of Corollary \ref{spectrumfinite}). If $(\alpha,\beta)\in E$ is a eigenvector associated with $\lambda_k$ (the $k$-th eigenvalue of the operator $M$)  then the following first order condition holds:
    \begin{eqnarray} 
\left\{
 \begin{array}{ll}
\alpha_i+\beta_i=\alpha_{i+1},\\ \displaystyle
\lambda_k\alpha_i=\frac{\beta_{i-1}}{\ell_{i-1}} - \frac{\beta_i}{\ell_i}, \quad \textnormal{for } 1\leqslant i  \leqslant n.
 \end{array}
  \right. 
  \label{sysT}
 \end{eqnarray}
   \end{proposition}
 
 \begin{proof} The first equation $\alpha_i+\beta_i=\alpha_{i+1}$ in \eqref{sysT} is nothing but the fact that $(\alpha,\beta)\in E$. Let us prove the second assertion, that comes from the minimisation of the Rayleigh quotient, as in the proof of Proposition \ref{firstorder1}. We can assume that $k\geqslant 2$ because the case of $k=1$ is given by  Proposition \ref{firstorder1}. Then we denote by $E_k$ the direct sum of the first $k$ eigenspaces associated to $\lambda_0,\lambda_1,\cdots,\lambda_{k-1}$  and use the min-max formula to write
 $$\lambda_k= \min_{u \in E_k^\bot, u\not = 0} \frac{Q(u)}{\|u\|_E^2}.$$
 Then for a fixed index $i\in\{1,\cdots,n\}$ we consider the element $w_i=(\partial_i, \partial_{i-1}-\partial_i)$, which belongs to the space $E$. To make sure that $w_i\in E_k^\bot$ we consider $p_i=\pi_{E_k}(w_i)$ so that   $w_i-p_i \in E_{k}^\bot$ and we can use it as a test function in the Rayleigh quotient. Therefore, if $u$ is a minimizer we must have, for all $t\in \R$, 
 $$\lambda_k =\frac{Q(u)}{\|u\|_E^2} \leqslant  \frac{Q(u+t (w_i-p_i))}{\|u+t (w_i-p_i)\|_E^2}.$$
 Since this must hold for all $t$, by a standard variational argument already used twice before we must have 
 $$\lambda_k \langle u, w_i-p_i \rangle_E = B(u,w_i-p_i),$$
 where $B$ is the bilinear form associated with $Q$. Or differently, with $u=(\alpha, \beta)$,
 \begin{eqnarray}
 \lambda_k \alpha_i -  \langle  u,p_i \rangle_E = B(u,w_i)-B(u,p_i)=\frac{\beta_{i-1}}{\ell_{i-1}} - \frac{\beta_i}{\ell_i}-B(u,p_i).\label{ident}
\end{eqnarray} 
  But now $\langle  u,p_i \rangle_E=0$ because  $p_i\in E_k$ and  $u \in E_k^\bot$. Similarly, we also have that  $B(u,p_i)=0$ because $u$ is an eigenfunction so that 
  $$B(u,p_i)= \langle  Mu,p_i \rangle_E=\lambda_k\langle  u,p_i \rangle_E=0.$$
  Returning to \eqref{ident}, we deduce that 
   \begin{eqnarray}
 \lambda_k \alpha_i =\frac{\beta_{i-1}}{\ell_{i-1}} - \frac{\beta_i}{\ell_i},
\end{eqnarray} 
which finishes the proof.
  \end{proof}

\subsection{Reilly inequality for polygons}


\

\

\begin{proof}[Proof of Proposition \ref{ReillyPolygon}] The proof follows the same scheme as the proof of Proposition \ref{reilly1}. Following the Definition \ref{deflambda1} we define the discreate measure 
 $$\mu_A:=\sum_{i=1}^k\delta_{A_i},$$
 and the vectorial function $${\bf H}(x)= \left(\frac{A_{i}-A_{i-1}}{|A_i-A_{i-1}|}+ \frac{A_{i}-A_{i+1}}{|A_i-A_{i+1}|}\right) {\bf 1}_{\{A_i\}}(x),$$
 in such a way that $\delta V_A= {\bf H}\mu_A$.
Since  $\mu_A$ has finite mass,    up to use a translation we can assume that
$$ \int_{\mathbb{R}^N}X \; d \mu_A=0.$$
 Then we take $X_i$ as a test function in the definition of $\lambda_1^A$ which yields,
 $$\lambda_1^A \int_{\mathbb{R}^N} (X_i)^2 \;d  \mu_A\leqslant \int |\nabla_P X_i|^2 \; dV(X,P)$$ 
 which after summation becomes
 \begin{eqnarray}
 \lambda_1^A \int_{\mathbb{R}^N} |X|^2  \;d \mu_A\leqslant \sum_{i=1}^{N}\int  |\nabla_P X_i|^2 \; dV(X,P). \label{ineq1}
 \end{eqnarray}
 Then we can use Cauchy-Schwarz (for vectors of $\mathbb{R}^N$), H\"older and  Hsiung-Minkowski \eqref{eq1} with $m=1$ and $M=\mathcal{H}^1(K_A)$ to obtain
\begin{eqnarray}
 \lambda_1^A \int_{\mathbb{R}^N} |X|^2  \;d  \mu_A &\leqslant& \sum_{i=1}^{N}\int |\nabla_P X_i|^2 \; dV(X,P)=  \int_{\mathbb{R}^N} X \cdot \nu \;d\|\delta V\|  \notag \\
 &=& \frac{\left( \int_{\mathbb{R}^N} X \cdot \nu \;d\|\delta V\|  \right)^2}{M} = \frac{\left(\int_{\mathbb{R}^N} X \cdot {\bf H}\;d \mu_A\right)^2}{M}  \label{eigen} \\
 &\leqslant &  \frac{ \left( \int_{\mathbb{R}^N} |X | |H | \;d \mu_A \right)^2}{M} \leqslant   \frac{\|X\|_{L^2(d\mu_A)}^2 \|{\bf H}\|_{L^2(d\mu_A)}^2 }{M} \notag
 \end{eqnarray}
 which proves \eqref{ineqRP}.
  \end{proof}

\subsection{Equality case for polygons}

\begin{proof}[Proof of Proposition \ref{Equality-Case}] If $A$ is a polygon that realizes the equality in \eqref{ineqRP}, then all the inequalities that have been used in the proof, are equalities. In particular the equality in Cauchy-Schwarz inequality says that $X$ must be colinear to $H_i$ at the point $A_i$. This implies the existence of a constant $c_i \in \R \setminus \{0\}$ such that 
\begin{eqnarray}
c_iA_i=H_i =\left(\frac{A_{i}-A_{i-1}}{|A_i-A_{i-1}|}+ \frac{A_{i}-A_{i+1}}{|A_i-A_{i+1}|}\right)\quad \quad \forall 1\leqslant i\leqslant n. \label{relation}
\end{eqnarray}
Let us fist prove that the polygon must be planar, in other words that  ${\rm dim} Vect\{ A_i\}=2$. To see this we pick an arbitrary $i$ and notice that \eqref{relation} implies 
$$A_i \in \textnormal{Vect}\{ A_i - A_{i-1} , A_i-A_{i+1}\}=:V.$$ But then since $A_{i+1}=A_i + (A_{i+1} -A_i)$ we deduce that $A_{i+1} \in V$. But since 
\begin{eqnarray}
c_{i+1}A_{i+1}=\left(\frac{A_{i+1}-A_{i}}{|A_{i+1}-A_{i}|}+ \frac{A_{i+1}-A_{i+2}}{|A_{i+1}-A_{i+2}|}\right) \notag,
\end{eqnarray}
we deduce that $A_{i+2}\in V$. We can continue this way recursively up to obtain that $A_i \in V$ for all $i$, thus we have proved that the polygon must be planar. 

As a result, we can assume for the rest of the proof that $N=2$. We know from the equality case in \eqref{eigen}, that for all $1\leqslant k \leqslant 2$, the coordinate function $X_k$ must be an eigenfunction.  In particular the infimum in the definition of $\lambda_1^A$\eqref{lambdaP}, is achieved, and is a minimum.  Moreover, the coordinate system of $\R^2$ being arbitrary, we deduce that for any fixed  vector $\nu \in \R^2$ with $|\nu|=1$, the projection $u:x\mapsto \langle \nu, x \rangle$ must be an eigenfunction. In particular from Proposition \ref{firstorder1} we know that  while   $T_i$ being the parameterization of $[A_i,A_{i+1}]$ defined by $T_i(t) := (1-t)A_i + tA_{i+1}$, then there exists some constants $\alpha_i,\beta_i \in \R$ such that 
$$
\left\{
 \begin{array}{ll}
u\circ T_i(t)=\alpha_i +t \beta_i, \quad  &\textnormal{ for } t\in[0,1] \textnormal{ and }  1\leqslant i \leqslant n,\\
\alpha_i+\beta_i=\alpha_{i+1},\\\displaystyle 
\lambda_1\alpha_i=\left(\frac{\beta_{i-1}}{\ell_{i-1}} - \frac{\beta_i}{\ell_i}\right), \quad &\textnormal{for } 1\leqslant i  \leqslant n.
 \end{array}
  \right. 
 $$
 where $\ell_i:=|A_{i+1}-A_i|$ and $\alpha_{n+1}=\alpha_1$ and $\beta_{n+1}=\beta_1$.  We consider  a given segment $[A_{i},A_{i+1}]$   and we apply the above with an arbitrary vector $\nu$ with $|\nu|=1$. Then
$$u\circ{T_i(t)}=\langle  \nu , (1-t)A_i + tA_{i+1} \rangle =  \langle \nu , A_i \rangle + t\langle \nu , A_{i+1}-A_i \rangle$$
so $\alpha_i =  \langle \nu , A_i \rangle \text{ and } \beta_i=\langle \nu , A_{i+1}-A_i \rangle.$
The relation $$\lambda_1\alpha_i=\frac{\beta_{i-1}}{\ell_{i-1}} - \frac{\beta_i}{\ell_i}$$ says 
$$\lambda_1 \langle \nu , A_i \rangle= \left\langle \nu ,  \frac{A_{i}-A_{i-1}}{|A_{i}-A_{i-1}|} + \frac{A_i - A_{i+1}}{|A_{i}-A_{i+1}|} \right\rangle=\langle \nu , H_i\rangle.$$
Remembering \eqref{relation}, and since $\nu$ is arbitrary, we actually get 
$$c_i=\lambda_1, \quad \textnormal{for }1\leqslant i\leqslant n. $$
A simple computation revels that  (see Proposition \ref{propPoly})
$$|H_i|^2=2(1-\cos(\theta_i))$$
where $\theta_i$ is the oriented angle between $\frac{A_{i}-A_{i-1}}{|A_{i}-A_{i-1}|}$ and $\frac{A_i - A_{i+1}}{|A_{i}-A_{i+1}|}$.
Since $\lambda_1A_i=H_i$ we deduce that 
\begin{eqnarray}
|A_i|^2=\frac{2(1-\cos(\theta_i))}{\lambda_1^2}.\label{normrelation}
\end{eqnarray}
Now we choose a particular $\nu$, which we decide to be equal to 
$$\nu:=\left(\frac{A_{i+1}-A_{i}}{|A_{i+1}-A_i|}\right)^{\bot}.$$ 
Here the symbol $\perp$ means the rotation by 90 degrees in the positive sense. Then 
\begin{eqnarray}
\lambda_1 \langle \nu , A_i \rangle= \left\langle \nu ,  \frac{A_{i}-A_{i-1}}{|A_{i}-A_{i-1}|} + \frac{A_i - A_{i+1}}{|A_{i}-A_{i+1}|} \right\rangle  =\left\langle \nu ,  \frac{A_{i}-A_{i-1}}{|A_i-A_{i-1}|} \right\rangle, \label{angle1}
\end{eqnarray}
because we decided to chose $\nu$ orthogonal to $A_i-A_{i+1}$. And we also have
\begin{eqnarray}
\lambda_1 \langle \nu , A_{i+1} \rangle= \left\langle \nu ,  \frac{A_{i+1}-A_{i}}{|A_{i+1}-A_{i}|} + \frac{A_{i+1} - A_{i+2}}{|A_{i+1}-A_{i+2}|} \right\rangle=\left\langle \nu ,  \frac{A_{i+1}-A_{i+2}}{|A_{i+1}-A_{i+2}|} \right\rangle.\label{angle2}
\end{eqnarray}
Now we can write, using again the orthogonality of $\nu$,
$$\langle \nu , A_{i+1} \rangle= \langle \nu , A_i \rangle +\langle \nu , A_{i+1}-A_i \rangle=\langle \nu,A_i \rangle.$$
Returning back to \eqref{angle1} and \eqref{angle2}, we have proved that 
\begin{align}\label{orth}
    \left\langle \nu ,  \frac{A_{i}-A_{i-1}}{|A_i-A_{i-1}|} \right\rangle=\left\langle \nu ,  \frac{A_{i+1}-A_{i+2}}{|A_{i+1}-A_{i+2}|} \right\rangle.
\end{align}
Now, we will prove that with this equality, there are only four types of polygons that can be constructed. We set
\begin{align*}
    w_i := \frac{A_{i+1}-A_{i}}{|A_{i+1}-A_{i}|},
\end{align*}
it is easy to see that if $\lambda_1A_i=H_i$, then the line $\R A_i$ is the bissector of the angle $\theta_i$, and the equality (\ref{orth}) implies either $\theta_{i+1} = \theta_i$ or $\theta_{i+1}= \pi -\theta_i$.
\begin{center}

\tikzset{every picture/.style={line width=0.75pt}} 

\tikzset{every picture/.style={line width=0.75pt}} 

\begin{tikzpicture}[x=0.75pt,y=0.75pt,yscale=-1,xscale=1]

\draw   (43.63,117.41) -- (94.63,45.86) -- (196.58,44.63) ;
\draw   (196.58,44.63) -- (264.97,44.63) -- (306,104.4) ;
\draw    (180,192.4) ;
\draw [shift={(180,192.4)}, rotate = 0] [color={rgb, 255:red, 0; green, 0; blue, 0 }  ][fill={rgb, 255:red, 0; green, 0; blue, 0 }  ][line width=0.75]      (0, 0) circle [x radius= 2.01, y radius= 2.01]   ;
\draw  [dash pattern={on 0.84pt off 2.51pt}]  (77,15.4) -- (180,192.4) ;
\draw  [dash pattern={on 0.84pt off 2.51pt}]  (180,192.4) -- (286,8.4) ;
\draw  [dash pattern={on 0.84pt off 2.51pt}]  (180,192.4) -- (463,26.29) ;
\draw   (264.97,44.63) -- (431.15,44.63) -- (389.33,101.4) ;
\draw    (94.63,45.86) ;
\draw [shift={(94.63,45.86)}, rotate = 0] [color={rgb, 255:red, 0; green, 0; blue, 0 }  ][fill={rgb, 255:red, 0; green, 0; blue, 0 }  ][line width=0.75]      (0, 0) circle [x radius= 2.01, y radius= 2.01]   ;
\draw [color={rgb, 255:red, 208; green, 2; blue, 27 }  ,draw opacity=1 ]   (264.97,44.63) -- (242.15,12.07) ;
\draw [shift={(241,10.43)}, rotate = 54.98] [color={rgb, 255:red, 208; green, 2; blue, 27 }  ,draw opacity=1 ][line width=0.75]    (10.93,-3.29) .. controls (6.95,-1.4) and (3.31,-0.3) .. (0,0) .. controls (3.31,0.3) and (6.95,1.4) .. (10.93,3.29)   ;
\draw [color={rgb, 255:red, 144; green, 19; blue, 254 }  ,draw opacity=1 ]   (264.97,44.63) -- (265,6.43) ;
\draw [shift={(265,4.43)}, rotate = 90.04] [color={rgb, 255:red, 144; green, 19; blue, 254 }  ,draw opacity=1 ][line width=0.75]    (10.93,-3.29) .. controls (6.95,-1.4) and (3.31,-0.3) .. (0,0) .. controls (3.31,0.3) and (6.95,1.4) .. (10.93,3.29)   ;
\draw  [dash pattern={on 4.5pt off 4.5pt}]  (43.63,117.41) -- (26,142.29) ;
\draw  [dash pattern={on 4.5pt off 4.5pt}]  (306,104.4) -- (324,129.29) ;
\draw  [dash pattern={on 4.5pt off 4.5pt}]  (389.33,101.4) -- (372,126.29) ;
\draw [color={rgb, 255:red, 208; green, 2; blue, 27 }  ,draw opacity=1 ]   (431.15,44.63) -- (454.83,11.91) ;
\draw [shift={(456,10.29)}, rotate = 125.89] [color={rgb, 255:red, 208; green, 2; blue, 27 }  ,draw opacity=1 ][line width=0.75]    (10.93,-3.29) .. controls (6.95,-1.4) and (3.31,-0.3) .. (0,0) .. controls (3.31,0.3) and (6.95,1.4) .. (10.93,3.29)   ;
\draw [color={rgb, 255:red, 144; green, 19; blue, 254 }  ,draw opacity=1 ]   (431.15,44.63) -- (431.18,6.43) ;
\draw [shift={(431.18,4.43)}, rotate = 90.04] [color={rgb, 255:red, 144; green, 19; blue, 254 }  ,draw opacity=1 ][line width=0.75]    (10.93,-3.29) .. controls (6.95,-1.4) and (3.31,-0.3) .. (0,0) .. controls (3.31,0.3) and (6.95,1.4) .. (10.93,3.29)   ;
\draw    (264.97,44.63) ;
\draw [shift={(264.97,44.63)}, rotate = 0] [color={rgb, 255:red, 0; green, 0; blue, 0 }  ][fill={rgb, 255:red, 0; green, 0; blue, 0 }  ][line width=0.75]      (0, 0) circle [x radius= 2.01, y radius= 2.01]   ;
\draw    (431.15,44.63) ;
\draw [shift={(431.15,44.63)}, rotate = 0] [color={rgb, 255:red, 0; green, 0; blue, 0 }  ][fill={rgb, 255:red, 0; green, 0; blue, 0 }  ][line width=0.75]      (0, 0) circle [x radius= 2.01, y radius= 2.01]   ;
\draw    (90,63.29) .. controls (104,65.29) and (108,63.29) .. (117,51.29) ;
\draw    (246,51.29) .. controls (254,63.29) and (255,65.29) .. (273,67.29) ;
\draw    (388,49.29) .. controls (388,62.29) and (390,69.29) .. (404,74.29) ;
\draw [color={rgb, 255:red, 144; green, 19; blue, 254 }  ,draw opacity=1 ]   (94.63,45.86) -- (94.66,7.66) ;
\draw [shift={(94.66,5.66)}, rotate = 90.04] [color={rgb, 255:red, 144; green, 19; blue, 254 }  ,draw opacity=1 ][line width=0.75]    (10.93,-3.29) .. controls (6.95,-1.4) and (3.31,-0.3) .. (0,0) .. controls (3.31,0.3) and (6.95,1.4) .. (10.93,3.29)   ;
\draw [color={rgb, 255:red, 208; green, 2; blue, 27 }  ,draw opacity=1 ]   (94.63,45.86) -- (116.89,12.66) ;
\draw [shift={(118,11)}, rotate = 123.84] [color={rgb, 255:red, 208; green, 2; blue, 27 }  ,draw opacity=1 ][line width=0.75]    (10.93,-3.29) .. controls (6.95,-1.4) and (3.31,-0.3) .. (0,0) .. controls (3.31,0.3) and (6.95,1.4) .. (10.93,3.29)   ;

\draw (159,194.4) node [anchor=north west][inner sep=0.75pt]    {$0$};
\draw (63,34.4) node [anchor=north west][inner sep=0.75pt]    {$A_{i}$};
\draw (196,-1.6) node [anchor=north west][inner sep=0.75pt]  [color={rgb, 255:red, 208; green, 2; blue, 27 }  ,opacity=1 ]  {$-w_{i+1}$};
\draw (271,-0.6) node [anchor=north west][inner sep=0.75pt]  [color={rgb, 255:red, 144; green, 19; blue, 254 }  ,opacity=1 ]  {$\nu $};
\draw (457,-2.6) node [anchor=north west][inner sep=0.75pt]  [color={rgb, 255:red, 208; green, 2; blue, 27 }  ,opacity=1 ]  {$-w_{i+1}$};
\draw (413,0.4) node [anchor=north west][inner sep=0.75pt]  [color={rgb, 255:red, 144; green, 19; blue, 254 }  ,opacity=1 ]  {$\nu $};
\draw (277,24.4) node [anchor=north west][inner sep=0.75pt]    {$A_{i+1}$};
\draw (393,24.4) node [anchor=north west][inner sep=0.75pt]    {$A_{i+1}$};
\draw (119,54.69) node [anchor=north west][inner sep=0.75pt]    {$\theta _{i}$};
\draw (217,54.4) node [anchor=north west][inner sep=0.75pt]    {$\theta _{i+1}$};
\draw (124,-1.6) node [anchor=north west][inner sep=0.75pt]  [color={rgb, 255:red, 208; green, 2; blue, 27 }  ,opacity=1 ]  {$w_{i-1}$};
\draw (350,52.4) node [anchor=north west][inner sep=0.75pt]    {$\theta _{i+1}$};
\draw (78,0.4) node [anchor=north west][inner sep=0.75pt]  [color={rgb, 255:red, 144; green, 19; blue, 254 }  ,opacity=1 ]  {$\nu $};

\end{tikzpicture}

\end{center}
\begin{center}
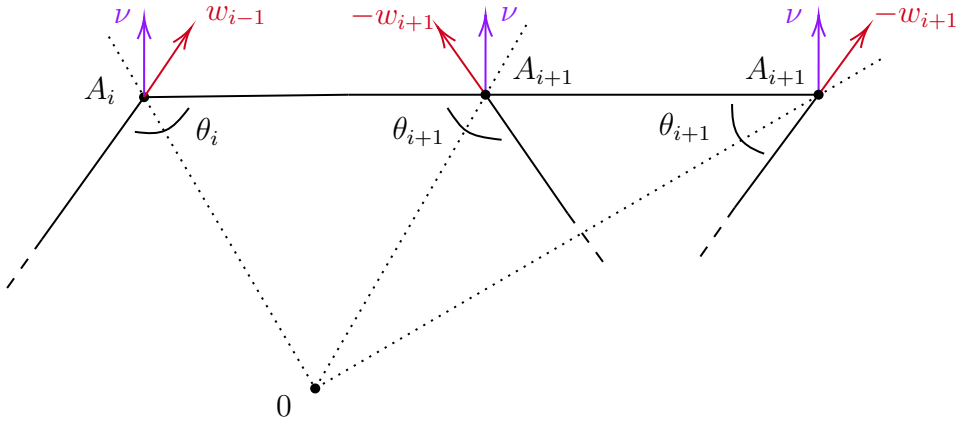
    Figure 1. The two possible cases of $A_{i+1}$.
\end{center}
We consider the oriented angles $\alpha_i$ between $A_i$ and $A_{i+1}$. The sum of the angles in a triangle being $\pi$, we know that $\alpha_i$ can take three types of values, $\theta_1$, $\pi-\theta_1$ or $\pi/2$. Furthermore, depending on the construction of $A_{i+1}$, we know that the values of $\alpha_i$ follow the graph below
\begin{center}

\tikzset{every picture/.style={line width=0.75pt}} 

\begin{tikzpicture}[x=0.75pt,y=0.75pt,yscale=-1,xscale=1]

\draw   (115,131) .. controls (115,117.19) and (126.19,106) .. (140,106) .. controls (153.81,106) and (165,117.19) .. (165,131) .. controls (165,144.81) and (153.81,156) .. (140,156) .. controls (126.19,156) and (115,144.81) .. (115,131) -- cycle ;
\draw   (236,69) .. controls (236,55.19) and (247.19,44) .. (261,44) .. controls (274.81,44) and (286,55.19) .. (286,69) .. controls (286,82.81) and (274.81,94) .. (261,94) .. controls (247.19,94) and (236,82.81) .. (236,69) -- cycle ;
\draw   (235,189) .. controls (235,175.19) and (246.19,164) .. (260,164) .. controls (273.81,164) and (285,175.19) .. (285,189) .. controls (285,202.81) and (273.81,214) .. (260,214) .. controls (246.19,214) and (235,202.81) .. (235,189) -- cycle ;
\draw   (356,131) .. controls (356,117.19) and (367.19,106) .. (381,106) .. controls (394.81,106) and (406,117.19) .. (406,131) .. controls (406,144.81) and (394.81,156) .. (381,156) .. controls (367.19,156) and (356,144.81) .. (356,131) -- cycle ;
\draw    (140,106) .. controls (163.76,71.35) and (178.7,66.1) .. (234.3,68.91) ;
\draw [shift={(236,69)}, rotate = 183.01] [color={rgb, 255:red, 0; green, 0; blue, 0 }  ][line width=0.75]    (10.93,-3.29) .. controls (6.95,-1.4) and (3.31,-0.3) .. (0,0) .. controls (3.31,0.3) and (6.95,1.4) .. (10.93,3.29)   ;
\draw    (286,69) .. controls (352.98,67.03) and (353.98,71.85) .. (379.8,104.49) ;
\draw [shift={(381,106)}, rotate = 231.55] [color={rgb, 255:red, 0; green, 0; blue, 0 }  ][line width=0.75]    (10.93,-3.29) .. controls (6.95,-1.4) and (3.31,-0.3) .. (0,0) .. controls (3.31,0.3) and (6.95,1.4) .. (10.93,3.29)   ;
\draw    (381,156) .. controls (363.18,190.65) and (350.26,192.96) .. (286.93,189.12) ;
\draw [shift={(285,189)}, rotate = 3.52] [color={rgb, 255:red, 0; green, 0; blue, 0 }  ][line width=0.75]    (10.93,-3.29) .. controls (6.95,-1.4) and (3.31,-0.3) .. (0,0) .. controls (3.31,0.3) and (6.95,1.4) .. (10.93,3.29)   ;
\draw    (235,189) .. controls (166.4,189) and (155.42,185.16) .. (140.89,157.71) ;
\draw [shift={(140,156)}, rotate = 62.65] [color={rgb, 255:red, 0; green, 0; blue, 0 }  ][line width=0.75]    (10.93,-3.29) .. controls (6.95,-1.4) and (3.31,-0.3) .. (0,0) .. controls (3.31,0.3) and (6.95,1.4) .. (10.93,3.29)   ;
\draw    (238,174) .. controls (214.48,160.28) and (228.42,113.91) .. (240.28,88.53) ;
\draw [shift={(241,87)}, rotate = 115.64] [color={rgb, 255:red, 0; green, 0; blue, 0 }  ][line width=0.75]    (10.93,-3.29) .. controls (6.95,-1.4) and (3.31,-0.3) .. (0,0) .. controls (3.31,0.3) and (6.95,1.4) .. (10.93,3.29)   ;
\draw    (278,89) .. controls (304.32,103.62) and (287.87,148.67) .. (279.62,171.3) ;
\draw [shift={(279,173)}, rotate = 289.98] [color={rgb, 255:red, 0; green, 0; blue, 0 }  ][line width=0.75]    (10.93,-3.29) .. controls (6.95,-1.4) and (3.31,-0.3) .. (0,0) .. controls (3.31,0.3) and (6.95,1.4) .. (10.93,3.29)   ;
\draw    (121,149) .. controls (47.87,142.17) and (93.58,122.04) .. (117.23,116.4) ;
\draw [shift={(119,116)}, rotate = 167.74] [color={rgb, 255:red, 0; green, 0; blue, 0 }  ][line width=0.75]    (10.93,-3.29) .. controls (6.95,-1.4) and (3.31,-0.3) .. (0,0) .. controls (3.31,0.3) and (6.95,1.4) .. (10.93,3.29)   ;
\draw    (401,145) .. controls (459.11,140.07) and (441.55,111.86) .. (400.87,111.98) ;
\draw [shift={(399,112)}, rotate = 358.64] [color={rgb, 255:red, 0; green, 0; blue, 0 }  ][line width=0.75]    (10.93,-3.29) .. controls (6.95,-1.4) and (3.31,-0.3) .. (0,0) .. controls (3.31,0.3) and (6.95,1.4) .. (10.93,3.29)   ;

\draw (246,59.4) node [anchor=north west][inner sep=0.75pt]    {$\pi /2$};
\draw (243,182.4) node [anchor=north west][inner sep=0.75pt]    {$\pi /2$};
\draw (129,124.4) node [anchor=north west][inner sep=0.75pt]    {$\theta _{0}$};
\draw (357,121.4) node [anchor=north west][inner sep=0.75pt]    {$\pi -\theta _{0}$};

\end{tikzpicture}

\end{center}
\begin{center}
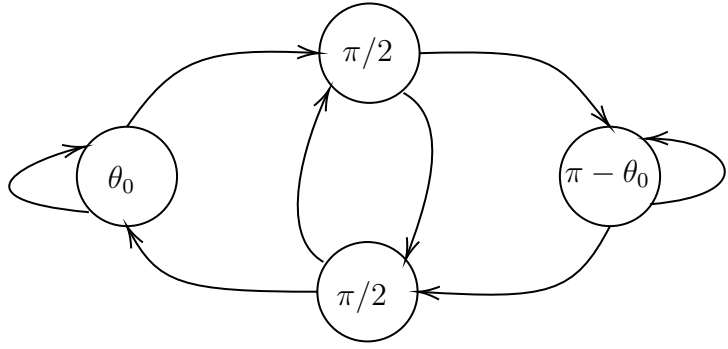
    Figure 2. Next possible values of $\alpha_i$.
\end{center}
The sum of $\alpha_i$ is equal to $2\pi$, so there exists $j,k,\ell \in \N$ such that
\begin{align*}
    j \theta_1 + k(\pi - \theta_1) + \ell\pi/2 = 2\pi,
\end{align*}
with $j \geqslant 1$. We distinguish the possible triplets $(j,k,\ell)$ based on $\ell$.

\emph{Case} $\ell=0$. Then $k=0$ because if $\alpha_{i}=\pi- \theta_0$, there exists $i_0$ such that $\alpha_{i_0}= \pi/2$ which is absurd. But this relation actually proves that $\theta_i =\theta_1$ since the index $i$ was arbitrary. Together with (\ref{relation}), we deduce that the points $A_i$ lie on a circle, and since
all the angles are equal we have a regular polygon.

\emph{Case} $\ell=1$. We started the construction of the polygon by $\alpha_i=\theta_1$ and we must to finish with $\theta_1$. Necessarily, this means that $\ell\geqslant 2$ which is absurd. This case is empty.

\emph{Case} $\ell=2$. The resulting equation is 
\begin{align*}
    (j-k)\theta_1 + k\pi = \pi,
\end{align*}
therefore we have two cases. The first one is $j=k=1$ and $n=4$. Using the fact that $\theta_1 = \theta_2$, and $\theta_3 = \theta_4$, with (\ref{relation}), we deduce that the points $A_1, A_2$ lie on a circle, and $A_3,A_4$ on another one. This characterizes the trapezoid. The second case is $k=0$. There is a rank $i_0$ such that $\theta_{i_0} = \pi-\theta_1$ and $\theta_i = \theta_1$ for all $i\ne i_0$. The equation (\ref{relation}) yields that all the points $A_i$, $i\ne i_0$, lie on a circle. This corresponds to the fake regular polygon.
\end{proof}

\subsection{Computation of $\lambda_1^A$ for a regular polygon.}

\medskip\begin{proposition}\label{Spectrum-Equilateral-Gon} Let $A=\{A_i\}_{1\leqslant i \leqslant n+1}$ be a polygon in $\R^N$ as in Definition \ref{defPolygon}. Assume that $A$ is equilateral and $L$ is the perimeter of $A$. Then  the spectrum of $A$ is 
$$\textnormal{Spec}(A)=\left\{\frac{4n}{L}\sin^2\frac{k\pi}{n}\ ; \ k\in\left\{0,\cdots\lfloor\dfrac{n}{2}\rfloor\right\}\right\}.$$
In particular,
$$\lambda_1^A=\frac{4n}{L}\sin^2\left(\frac{\pi}{n}\right).$$
\end{proposition}

\begin{proof} Since $A$ is equilateral, for any $i\in\{1,\cdots,n\}$, $|A_{i+1}-A_i|=\ell=L/n$. Let $u$ be an eigenfunction for an eigenvalue $\lambda$. From Proposition \ref{firstorder1} we know that if  $T_i$ is the parameterization of $[A_i,A_{i+1}]$ defined by $T_i(t) := (1-t)A_i + tA_{i+1}$, then there exists some constants $\alpha_i,\beta_i \in \R$ such that $u\circ T_i(t)=\alpha_i +t \beta_i$ for $t\in[0,1]$ and $ 1\leqslant i \leqslant n$, and 
$$
\left\{
 \begin{array}{ll}
\alpha_i+\beta_i=\alpha_{i+1},\\
\lambda\alpha_{i+1}=\dfrac{1}{\ell}(\beta_i-\beta_{i+1}), \quad \textnormal{ for } 1\leqslant i  \leqslant n,
 \end{array}
  \right. 
 $$
with $\alpha_{n+1}=\alpha_1$ and $\beta_{n+1}=\beta_1$. Note that this system above is equivalent to
\begin{equation}\label{equivalence-system}
 \begin{pmatrix}
\alpha_{i+1}\\\beta_{i+1}
\end{pmatrix}=M\begin{pmatrix}
\alpha_i\\\beta_i
\end{pmatrix}:=\begin{pmatrix}
1&1\\-\Lambda&1-\Lambda
\end{pmatrix}\begin{pmatrix}
\alpha_i\\\beta_i
\end{pmatrix},\quad \textnormal{ for }1\leqslant i  \leqslant n,
\end{equation}
with $\Lambda = \lambda \ell$. Moreover, $$M^n\begin{pmatrix}
\alpha_1\\\beta_1
\end{pmatrix}=\begin{pmatrix}
\alpha_1\\\beta_1
\end{pmatrix}.$$ Then $1$ is an eigenvalue of $M^n$ and it follows that the spectrum of $M$ contains an $n$-th root of $1$. A straightforward computation shows that the characteristic polynom of $M$ is $P=(X-1)^2+\Lambda(X-1)+\Lambda$ and the discriminant of $Y^2+\Lambda Y+\Lambda$ is $\Delta=\Lambda(\Lambda-4)$. Look for the conditions on $\Lambda$ so that the spectrum of $M$ contains an $n$-th root of $1$. 

\emph{Case} $\Lambda=0$. Imediatly, $\lambda=0$ and $\beta_i$ are constant. Then $\alpha_{n+1}=\alpha_1+n\beta_1=\alpha_1$, $\beta_i=0$ and $\alpha_i$ are constant. It follows that $u$ is constant. Conversely, a constant function is an eigenfunction for $\lambda=0$. 

 \emph{Case} $0<\Lambda<4$. One has $\Delta<0$ and the roots of $P$ are $x^{\pm}=(2-\Lambda)/2\pm i\sqrt{\Lambda(4-\Lambda)}/2$. It is easy to prove that $|x^{\pm}|=1$ and $x^+$ or $x^-=\overline{x^+}$ is an $n$-th root of $1$ if and only if $x^+$ or $\overline{x^+}=\exp(i2k\pi/n)$ for $1 \leqslant k \leqslant n-1$, in fact $k\neq0$ since $\Lambda\neq0$. This is  equivalent to say that $(2-\Lambda)/2=\cos(2k\pi/n)$ and $\Lambda=4\sin^2(k\pi/n)$.  From this we obtain that $$\lambda=\dfrac{4n}{L}\sin^2\dfrac{k\pi}{n}.$$ 
 
\emph{Case }$\Lambda=4$. We have $x=-1$ and $x$ is an $n$-th root of $1$ if and only if $n$ is even. In this case $$\lambda=\dfrac{4n}{L}=\dfrac{4n}{L}\sin^2\dfrac{(n/2)\pi}{n}.$$

\emph{Case } $\Lambda>4$. The roots $x^{\pm}\in\R$. If $x^{\pm}$ is an $n$-th root of $1$ then $x^{\pm}=\pm 1$ which is impossible. Thus  we have proved that the spectrum of $A$ is including in 
$$\textnormal{Spec}(A)\subset \left\{\frac{4n}{L}\sin^2\frac{k\pi}{n}\ \vert \ 0\leqslant k \leqslant n-1\right\}=\left\{\frac{4n}{L}\sin^2\frac{k\pi}{n}\ \vert\ 0 \leqslant k \leqslant \lfloor\dfrac{n}{2}\rfloor\right\}.$$
Conversely, let $k\in\{1,\cdots,n-1\}$, $\lambda=(4n/L)\sin^2(k\pi/n)$ and $\Lambda=\lambda\ell$ (the case $k=0$ corresponds to $\lambda=0$ already discussed above). Then $\exp(i2k\pi/n)$ is an eigenvalue of $M$ and $1$ is an eigenvalue of $M^n$. Let $(
\alpha_1,\beta_1)$ be an eigenvector of $M^n$ and for any $i\in\{2,\cdots,n+1\}$ define $$\begin{pmatrix}
\alpha_i\\\beta_i
\end{pmatrix}=M^{i-1}\begin{pmatrix}
\alpha_1\\\beta_1
\end{pmatrix}.$$ Then for any $i\in\{1,\cdots,n\}$, $$ \begin{pmatrix}
\alpha_{i+1}\\\beta_{i+1}
\end{pmatrix}=M\begin{pmatrix}
\alpha_i\\\beta_i
\end{pmatrix}.$$ If we define a function $u$ as above with scalar $\alpha_i$ and $\beta_i$, the equivalence \ref{equivalence-system} allows us to say $u$ is an eigenfunction for $\lambda=\frac{4n}{L}\sin^2\frac{k\pi}{n}$. Consequently, $$\textnormal{Spec}(A)=\left\{\frac{4n}{L}\sin^2\frac{k\pi}{n}\ ; \ k\in\left\{0,\cdots\lfloor\dfrac{n}{2}\rfloor\right\}\right\}.$$
\end{proof}


 \subsection{Computation of $\lambda_1^A$ for a trapeze.}

 In this section we compute the spectrum for a regular trapeze as defined below.
 \medskip\begin{definition}\label{definitionT} Let $\theta \in (0,\pi/2)$ be a given angle. We define the trapeze $T_\theta$ as being the polygon with the following vertices (see Figure 3)
 \begin{align*}
     A_1&=(\cos(\theta), \sin(\theta) ),  \quad &A_2&=(-\cos(\theta), \sin(\theta) ),\\
     A_3&=(-\tan(\theta)\sin(\theta) , -\tan(\theta)\sin(\theta)), \quad &A_4&=(\tan(\theta)\sin(\theta) , -\tan(\theta)\sin(\theta)).
 \end{align*}
 \end{definition}
 
 \medskip\begin{proposition} \label{calculLambda1}Let $T_\theta$ be the trapeze as defined in Definition \ref{definitionT}.  Then the spectrum of $T_\theta$ is given by the following two non trivial eigenvalues:
 $$\lambda_1^A(T_\theta)=2\cos(\theta), \textnormal{ and }
\lambda_2^A(T_\theta)=\frac{1}{\cos(\theta)\sin(\theta)^2}.$$
 Moreover, $\lambda_1^A(T_\theta)$ is double, and $\lambda_2^A(T_\theta)$ is simple.
 \end{proposition}
 
 \begin{center}
\tikzset{every picture/.style={line width=0.75pt}} 

\begin{tikzpicture}[x=0.75pt,y=0.75pt,yscale=-1,xscale=1]

\draw   (98,184) -- (130.1,77) -- (199.9,77) -- (232,184) -- cycle ;
\draw  [dash pattern={on 0.84pt off 2.51pt}]  (98,184) -- (166,109) ;
\draw  [dash pattern={on 0.84pt off 2.51pt}]  (166,109) -- (232,184) ;
\draw    (130.1,77) -- (166,109) ;
\draw  [dash pattern={on 0.84pt off 2.51pt}]  (199.9,77) -- (166,109) ;
\draw   (160.93,114.07) -- (156.27,109.41) -- (161.34,104.34) ;
\draw   (171.24,104.1) -- (175.74,108.91) -- (170.5,113.81) ;
\draw    (130.1,77) ;
\draw [shift={(130.1,77)}, rotate = 0] [color={rgb, 255:red, 0; green, 0; blue, 0 }  ][fill={rgb, 255:red, 0; green, 0; blue, 0 }  ][line width=0.75]      (0, 0) circle [x radius= 1.34, y radius= 1.34]   ;
\draw    (98,184) ;
\draw [shift={(98,184)}, rotate = 0] [color={rgb, 255:red, 0; green, 0; blue, 0 }  ][fill={rgb, 255:red, 0; green, 0; blue, 0 }  ][line width=0.75]      (0, 0) circle [x radius= 1.34, y radius= 1.34]   ;
\draw    (232,184) ;
\draw [shift={(232,184)}, rotate = 0] [color={rgb, 255:red, 0; green, 0; blue, 0 }  ][fill={rgb, 255:red, 0; green, 0; blue, 0 }  ][line width=0.75]      (0, 0) circle [x radius= 1.34, y radius= 1.34]   ;
\draw    (199.9,77) ;
\draw [shift={(199.9,77)}, rotate = 0] [color={rgb, 255:red, 0; green, 0; blue, 0 }  ][fill={rgb, 255:red, 0; green, 0; blue, 0 }  ][line width=0.75]      (0, 0) circle [x radius= 1.34, y radius= 1.34]   ;
\draw    (145.38,87.67) .. controls (148.11,85.67) and (148.11,85) .. (148.05,80.33) ;
\draw    (128.11,92.33) .. controls (134.78,93.67) and (134.78,93.67) .. (140.11,90.33) ;

\draw (103,57.4) node [anchor=north west][inner sep=0.75pt]    {$A_{2}$};
\draw (205,57.4) node [anchor=north west][inner sep=0.75pt]    {$A_{1}$};
\draw (75,182.4) node [anchor=north west][inner sep=0.75pt]    {$A_{3}$};
\draw (240,183.4) node [anchor=north west][inner sep=0.75pt]    {$A_{4}$};
\draw (146.33,97.73) node [anchor=north west][inner sep=0.75pt]  [font=\footnotesize]  {$1$};
\draw (130.11,95.73) node [anchor=north west][inner sep=0.75pt]  [font=\footnotesize]  {$\theta $};
\draw (150.67,79.4) node [anchor=north west][inner sep=0.75pt]  [font=\footnotesize]  {$\theta $};

\end{tikzpicture}

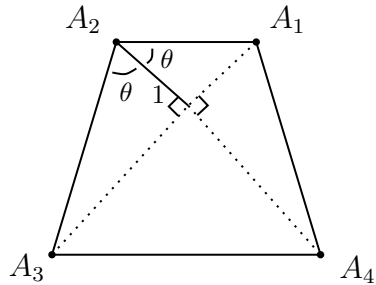
 Figure 3. The trapeze $T_\theta$.
 \end{center}
 
 \begin{proof} Let $u$ be an eigenfunction, which is piecewise affine. In particular from Proposition \ref{firstorder2} we know that   there exists some constants $\alpha_i,\beta_i \in \R$ such that 
\begin{eqnarray}\label{relation1}
\left\{
 \begin{array}{ll}
u\circ T_i(t)=\alpha_i +t \beta_i,\quad  \textnormal{ for } t\in[0,1] \text{ and }  1\leqslant i \leqslant 4,\\
\alpha_i+\beta_i=\alpha_{i+1},\\ \displaystyle
\lambda_1\alpha_i=\frac{\beta_{i-1}}{\ell_{i-1}} - \frac{\beta_i}{\ell_i}, \quad \textnormal{for } 1\leqslant i  \leqslant 4,
 \end{array}
  \right. 
 \end{eqnarray}
 where as usual $\ell_i:=|A_{i+1}-A_i|$. Here    $T_i$ is as before the parameterization of $[A_i,A_{i+1}]$ defined by $T_i(t) := (1-t)A_i + tA_{i+1}$.  The relation in \eqref{relation1} says  in particular that the values of $\alpha_i$ and $\beta_i$ are all determined by the values of $\alpha_1$ and $\beta_1$. Indeed, if $\beta_i$ and $\alpha_i$ are given then $\beta_{i+1}$ and $\alpha_{i+1}$  must satisfy
 $$\alpha_{i+1}=\alpha_i+\beta_i,  \textnormal{ and }
 \beta_{i+1}=\left(\frac{\ell_{i+1}}{\ell_i}-\lambda_1 \ell_{i+1}\right)\beta_i -\lambda_1\ell_{i+1} \alpha_i.$$
This suggests to introduce the following matrix, depending on the unknown variable $\lambda \in \R^+$, 
$$
M_i(\lambda):=
\left(
\begin{array}{cc}
1 & 1 \\
-\lambda \ell_{i+1} & \displaystyle \frac{\ell_{i+1}}{\ell_i}-\lambda_1 \ell_{i+1}\\
\end{array}
\right),
$$ 
in such a way  that
 $$
\left(
\begin{array}{c}
\alpha_{i+1} \\
\beta_{i+1}\\
\end{array}
\right)
=M_i(\lambda) \left(
\begin{array}{c}
\alpha_{i} \\
\beta_{i}\\
\end{array}
\right).
 $$
Now by periodicity, we must have $(\alpha_{5},\beta_5)=(\alpha_1,\beta_1)$, in other words 
$$\left(\prod_{i=1}^4 M_i(\lambda) \right) \left(
\begin{array}{c}
\alpha_{1} \\
\beta_{1}\\
\end{array}
\right)=
\left(\begin{array}{c}
\alpha_{1} \\
\beta_{1}\\
\end{array}
\right).
$$
This means that the vector $(\alpha_1,\beta_1)$ must be a eigenvector for the matrix 
$$M(\lambda):=\prod_{i=1}^4 M_i(\lambda),$$
associated to the eigenvalue $1$.  Of differently, the eigenvalues of $T_\theta$ must be a root of  the following equation
$${\rm det } (M(\lambda)-\textnormal{Id})=0.$$
Now the matrix $M(\lambda)$ can be computed explicitly. It only depends on the values of $\ell_i$ that are
\begin{align*}
    \ell_1= 2\cos(\theta), \quad   \ell_2=\ell_4=1/\cos(\theta), \quad
\ell_3= 2\tan(\theta)\sin(\theta).
\end{align*}
A direct computation revels that 
$$f(\lambda):={\rm det } (M(\lambda)-\textnormal{Id})=
\frac{4 \lambda (\lambda - 2\cos(\theta))^2(\lambda \cos(\theta)^3 - \lambda \cos(\theta) + 1))}{\cos(\theta)^3}.
$$
For the convenience of the reader, we provide a little MATLAB code in the appendix which produces this computation, and finds the solutions of the equation $f(\lambda)=0$.  The solutions are  
$$
\lambda=0, \text{ or } \lambda=2\cos(\theta),  \text{ or }  \lambda=\frac{1}{\cos(\theta)\sin(\theta)^2}.
$$
As it is readily checked that 
$$\frac{1}{\cos(\theta)\sin(\theta)^2} > 2\cos(\theta),$$
we deduce that $2\cos(\theta)$ is the first eigenvalue. Finally, substituting the value $\lambda=2\cos(\theta)$ in $M(\lambda)$ we find that 
$$M(2\cos(\theta))=
\left(
\begin{array}{cc}
1 & 0 \\
0  &  1\\
\end{array}
\right).
$$
This means that any vector $(\alpha_1,\beta_1)$ will give an eigenfunction, or differently, the eigenspace associated with $\lambda_1^A(T_\theta)$ is double. On the other hand, the matrix $M(\lambda_2^A(T_\theta))$ has determinant $1$ and since $1$ is an eigenvalue, both egeinvalues must be equal to $1$.  This matrix is either the identity, of equal to
$$
\left(
\begin{array}{cc}
1 & 1 \\
0  &  1\\
\end{array}
\right)
$$
in a suitable basis. A direct computation revels that $M(\lambda_2^A(T_\theta))$ is different from the identity matrix, which implies that we are in the second situation. Therefore, the eigenspace associated to $\lambda_2^A(T_\theta)$ is simple. This finishes the proof of the proposition. 
 \end{proof}

\medskip\begin{corollary} The trapeze $T_\theta$   defined in Definition \ref{definitionT} satisfies the equality case in Reilly's inequality, in other words the following equality holds
$$\lambda_1^A(T_\theta)=\frac{\sum_{i=1}^4|H_i|^2}{P(T_\theta)},$$
where $P(T_\theta)$ is the perimeter and $H_i$ are the curvature vectors. 
\end{corollary}

\begin{proof} We already know from Proposition \ref{calculLambda1} that $\lambda_1^A(T_\theta)=2\cos(\theta)$.  To compute the perimeter let us recall that
\begin{align*}
    \ell_1= 2\cos(\theta), \quad   \ell_2=\ell_4=1/\cos(\theta), \quad
\ell_3= 2\tan(\theta)\sin(\theta).
\end{align*}
This means that 
\begin{align*}
    P(T_\theta)&=2\cos(\theta) +  2/\cos(\theta) +2\tan(\theta)\sin(\theta) \notag = \frac{2\cos^2(\theta) +  2 +2 \sin^2(\theta)}{\cos(\theta)} = \frac{4}{\cos(\theta)}.
\end{align*}
Finally, the angles of $T_\theta$ at each vertex is either  $2\theta$ (at $A_1$ and $A_2$) or $\pi-2\theta$ (at $A_3$ and $A_4$). This means that $|H_i|= 2\cos(\theta)$ for $i=1,2$, and $|H_i|=2\cos(\pi/2-\theta)=2\sin(\theta)$ for $i=3,4$. Consequently,
$$\sum_{i=1}^4 |H_i|^2=8 \cos^2(\theta) + 8\sin^2(\theta)=8.$$
We conclude that 
$$\frac{1}{P(T_\theta)}\displaystyle\sum_{i=1}^4|H_i|^2=\frac{8\cos(\theta)}{4}=2\cos(\theta),$$
as desired.
\end{proof}

\subsection{Computation of $\lambda_1^A$ for a $n+2$-``fake regular" polygon.}
\label{fakeregular}
Let $n\geqslant2$ be an integer and $F_n=\{A_i\}_{0\leqslant i\leqslant n+1}$ be a $n+2$-``fake-regular'' polygon built as follows. Let $O$ and $A_1$ be points such that $OA_1=1$ and for any $k\in\{2,\cdots,n+1\}$, $A_k=r_{(k-1),n}(A_1)$ where $r_{(k-1),n}$ is the rotation of center $O$ and angle $(k-1)\pi/n$. Then for any $k\in\{2,\cdots,n\}$, $\widehat{A_2}=\cdots=\widehat{A_n}=\theta_n=\left(1-1/n\right)\pi$. Now we fix the point $A_0$ on the mediator of the segment $[A_1A_{n+1}]$ such that $\widehat{A_1}=\widehat{A_{n+1}}=\theta_n$. Then $\widehat{A_0}=\pi-\theta_n$ and easy computations show that for all $k\in\{1,\cdots,n\}$,
$$A_kA_{k+1}=a_n=2\sin\left(\dfrac{\pi}{2n}\right),$$
and
$$A_0A_1=A_0A_{n+1}=b_n=\dfrac{1}{\cos\left(\theta_n/2\right)}=\dfrac{2}{a_n}.$$
Now since for any $k\in\{1,\cdots,n+1\}$, $OA_k=1$, we have $H_k=2\cos\left(\theta_n/2\right)\overrightarrow{OA_k}.$ Moreover,
$$H_0=2\cos\left(\dfrac{\pi-\theta_n}{2}\right)\frac{\overrightarrow{OA_0}}{OA_0}=2\sin\left(\theta_n/2\right)\frac{\overrightarrow{OA_0}}{OA_0}$$

\tikzset{every picture/.style={line width=0.75pt}} 

\begin{tikzpicture}[x=0.75pt,y=0.75pt,yscale=-1,xscale=1]

\draw   (234,152.46) -- (222.14,196.73) -- (189.73,229.14) -- (145.46,241) -- (101.19,229.14) -- (68.78,196.73) -- (56.92,152.46) -- (68.78,108.19) -- (101.19,75.78) -- (145.46,63.92) -- (189.73,75.78) -- (222.14,108.19) -- cycle ;
\draw    (145.46,63.92) -- (477.92,152.75) ;
\draw    (145.46,241) -- (477.92,152.75) ;
\draw  [color={rgb, 255:red, 0; green, 0; blue, 0 }  ,draw opacity=0.19 ] (56.92,152.46) .. controls (56.92,103.56) and (96.56,63.92) .. (145.46,63.92) .. controls (194.36,63.92) and (234,103.56) .. (234,152.46) .. controls (234,201.36) and (194.36,241) .. (145.46,241) .. controls (96.56,241) and (56.92,201.36) .. (56.92,152.46) -- cycle ;
\draw [color={rgb, 255:red, 0; green, 0; blue, 0 }  ,draw opacity=0.33 ]   (146.92,301.75) -- (144.93,10.75) ;
\draw [shift={(144.92,8.75)}, rotate = 89.61] [color={rgb, 255:red, 0; green, 0; blue, 0 }  ,draw opacity=0.33 ][line width=0.75]    (10.93,-3.29) .. controls (6.95,-1.4) and (3.31,-0.3) .. (0,0) .. controls (3.31,0.3) and (6.95,1.4) .. (10.93,3.29)   ;
\draw [color={rgb, 255:red, 0; green, 0; blue, 0 }  ,draw opacity=0.33 ]   (4.92,152.75) -- (535.92,152.75) ;
\draw [shift={(537.92,152.75)}, rotate = 180] [color={rgb, 255:red, 0; green, 0; blue, 0 }  ,draw opacity=0.33 ][line width=0.75]    (10.93,-3.29) .. controls (6.95,-1.4) and (3.31,-0.3) .. (0,0) .. controls (3.31,0.3) and (6.95,1.4) .. (10.93,3.29)   ;
\draw    (145.46,63.92) ;
\draw [shift={(145.46,63.92)}, rotate = 0] [color={rgb, 255:red, 0; green, 0; blue, 0 }  ][fill={rgb, 255:red, 0; green, 0; blue, 0 }  ][line width=0.75]      (0, 0) circle [x radius= 1.34, y radius= 1.34]   ;
\draw    (189.73,75.78) ;
\draw [shift={(189.73,75.78)}, rotate = 0] [color={rgb, 255:red, 0; green, 0; blue, 0 }  ][fill={rgb, 255:red, 0; green, 0; blue, 0 }  ][line width=0.75]      (0, 0) circle [x radius= 1.34, y radius= 1.34]   ;
\draw    (101.19,75.78) ;
\draw [shift={(101.19,75.78)}, rotate = 0] [color={rgb, 255:red, 0; green, 0; blue, 0 }  ][fill={rgb, 255:red, 0; green, 0; blue, 0 }  ][line width=0.75]      (0, 0) circle [x radius= 1.34, y radius= 1.34]   ;
\draw    (56.92,152.46) ;
\draw [shift={(56.92,152.46)}, rotate = 0] [color={rgb, 255:red, 0; green, 0; blue, 0 }  ][fill={rgb, 255:red, 0; green, 0; blue, 0 }  ][line width=0.75]      (0, 0) circle [x radius= 1.34, y radius= 1.34]   ;
\draw    (68.78,108.19) ;
\draw [shift={(68.78,108.19)}, rotate = 0] [color={rgb, 255:red, 0; green, 0; blue, 0 }  ][fill={rgb, 255:red, 0; green, 0; blue, 0 }  ][line width=0.75]      (0, 0) circle [x radius= 1.34, y radius= 1.34]   ;
\draw    (68.78,196.73) ;
\draw [shift={(68.78,196.73)}, rotate = 0] [color={rgb, 255:red, 0; green, 0; blue, 0 }  ][fill={rgb, 255:red, 0; green, 0; blue, 0 }  ][line width=0.75]      (0, 0) circle [x radius= 1.34, y radius= 1.34]   ;
\draw    (101.19,229.14) ;
\draw [shift={(101.19,229.14)}, rotate = 0] [color={rgb, 255:red, 0; green, 0; blue, 0 }  ][fill={rgb, 255:red, 0; green, 0; blue, 0 }  ][line width=0.75]      (0, 0) circle [x radius= 1.34, y radius= 1.34]   ;
\draw    (145.46,241) ;
\draw [shift={(145.46,241)}, rotate = 0] [color={rgb, 255:red, 0; green, 0; blue, 0 }  ][fill={rgb, 255:red, 0; green, 0; blue, 0 }  ][line width=0.75]      (0, 0) circle [x radius= 1.34, y radius= 1.34]   ;
\draw    (189.73,229.14) ;
\draw [shift={(189.73,229.14)}, rotate = 0] [color={rgb, 255:red, 0; green, 0; blue, 0 }  ][fill={rgb, 255:red, 0; green, 0; blue, 0 }  ][line width=0.75]      (0, 0) circle [x radius= 1.34, y radius= 1.34]   ;
\draw    (222.14,108.19) ;
\draw [shift={(222.14,108.19)}, rotate = 0] [color={rgb, 255:red, 0; green, 0; blue, 0 }  ][fill={rgb, 255:red, 0; green, 0; blue, 0 }  ][line width=0.75]      (0, 0) circle [x radius= 1.34, y radius= 1.34]   ;
\draw    (234,152.46) ;
\draw [shift={(234,152.46)}, rotate = 0] [color={rgb, 255:red, 0; green, 0; blue, 0 }  ][fill={rgb, 255:red, 0; green, 0; blue, 0 }  ][line width=0.75]      (0, 0) circle [x radius= 1.34, y radius= 1.34]   ;
\draw    (222.14,196.73) ;
\draw [shift={(222.14,196.73)}, rotate = 0] [color={rgb, 255:red, 0; green, 0; blue, 0 }  ][fill={rgb, 255:red, 0; green, 0; blue, 0 }  ][line width=0.75]      (0, 0) circle [x radius= 1.34, y radius= 1.34]   ;
\draw    (477.92,152.75) ;
\draw [shift={(477.92,152.75)}, rotate = 0] [color={rgb, 255:red, 0; green, 0; blue, 0 }  ][fill={rgb, 255:red, 0; green, 0; blue, 0 }  ][line width=0.75]      (0, 0) circle [x radius= 1.34, y radius= 1.34]   ;
\draw    (145.46,152.46) ;
\draw [shift={(145.46,152.46)}, rotate = 0] [color={rgb, 255:red, 0; green, 0; blue, 0 }  ][fill={rgb, 255:red, 0; green, 0; blue, 0 }  ][line width=0.75]      (0, 0) circle [x radius= 1.34, y radius= 1.34]   ;
\draw [color={rgb, 255:red, 0; green, 0; blue, 0 }  ,draw opacity=0.33 ][fill={rgb, 255:red, 0; green, 0; blue, 0 }  ,fill opacity=0.33 ]   (101.19,75.78) -- (145.46,152.46) ;
\draw [color={rgb, 255:red, 0; green, 0; blue, 0 }  ,draw opacity=0.33 ]   (128.4,122.79) .. controls (131.63,118.94) and (140.56,116.17) .. (145.63,118.33) ;
\draw [color={rgb, 255:red, 0; green, 0; blue, 0 }  ,draw opacity=0.33 ]   (132.01,68.15) .. controls (137.44,74.72) and (152.01,74.15) .. (160.59,67.87) ;
\draw [color={rgb, 255:red, 0; green, 0; blue, 0 }  ,draw opacity=0.33 ]   (431.13,164.84) .. controls (421.99,159.13) and (423.13,145.13) .. (432.84,140.56) ;

\draw (146,47.4) node [anchor=north west][inner sep=0.75pt]  [font=\small]  {$A_{1} =\ B_{1} =B_{2n+1}$};
\draw (175.33,78.4) node [anchor=north west][inner sep=0.75pt]  [font=\small]  {$B_{2n}$};
\draw (38,58.4) node [anchor=north west][inner sep=0.75pt]  [font=\small]  {$A_{2} \ =\ B_{2}$};
\draw (38,230.4) node [anchor=north west][inner sep=0.75pt]  [font=\small]  {$A_{n} \ =\ B_{n}$};
\draw (147.46,244.4) node [anchor=north west][inner sep=0.75pt]  [font=\small]  {$A_{n+1} \ =\ B_{n+1}$};
\draw (154.4,211.2) node [anchor=north west][inner sep=0.75pt]  [font=\small]  {$B_{n+2}$};
\draw (129,157.4) node [anchor=north west][inner sep=0.75pt]    {$0$};
\draw (116.95,71.4) node [anchor=north west][inner sep=0.75pt]  [font=\small]  {$a_{n}$};
\draw (110.67,108.73) node [anchor=north west][inner sep=0.75pt]  [font=\footnotesize]  {$1$};
\draw (149.33,75.4) node [anchor=north west][inner sep=0.75pt]  [font=\footnotesize]  {$\theta _{n}$};
\draw (128,91.4) node [anchor=north west][inner sep=0.75pt]  [font=\small]  {$\frac{\pi }{n}$};
\draw (372.67,136.35) node [anchor=north west][inner sep=0.75pt]  [font=\small]  {$\pi -\theta _{n}$};
\draw (472.06,132.01) node [anchor=north west][inner sep=0.75pt]    {$A_{0}$};
\draw (305.72,88.01) node [anchor=north west][inner sep=0.75pt]  [font=\small]  {$b_{n}$};

\end{tikzpicture}
\begin{center}
    Figure 2. A $(n+2)-$fake regular-polygone.
\end{center}

\vspace{0.5cm}
Now consider the $2n$-gon $\Gamma_{2n}=\{B_i\}_{1\leqslant i\leqslant n+1}$ where for any $k\in\{1,\cdots,n+1\}$, $B_k=A_k$ and for $k\in\{2,\cdots,n\}$, $B_{n+k}=r_{(k-1),n}(B_{n+1})$. By convention we put $B_{2n+1}=B_1$. The perimeter of $\Gamma_{2n}$ is $P(\Gamma_{2n})=2na_n=4n\sin\left(\pi/2n\right)$ and from the proposition \ref{Spectrum-Equilateral-Gon} the first eigenvalue $\lambda_1^{\Gamma_{2n}}$ of $\Gamma_{2n}$ is 
\begin{equation}\label{Spectrum-Gamma-2n}\lambda_1^{\Gamma_{2n}}=\dfrac{8n}{P(\Gamma_{2n})}\sin^2\left(\dfrac{\pi}{2n}\right)=2\sin\left(\dfrac{\pi}{2n}\right)=a_n.
\end{equation}

\medskip\begin{proposition} The first nonzero eigenvalue $\lambda_1^{F_n}$ is 
$$\lambda_1^{F_n}=\lambda_1^{\Gamma_{2n}}=\frac{1}{P(F_n)}\displaystyle\sum_{i=0}^{n+1}|H_i|^2=a_n.$$
\end{proposition}

\begin{proof} First, we have
\begin{align*}
\dfrac{1}{P(F_n)}\displaystyle\sum_{i=0}^{n+1}|H_i|^2&=\dfrac{4(n+1)\cos^2\left(\theta_n/2\right)+4\sin^2\left(\theta_n/2\right)}{2n\sin\left(\dfrac{\pi}{2n}\right)+\dfrac{2}{\sin\left(\pi/2n\right)}}\\
&=2\dfrac{n\cos^2\left(\theta_n/2\right)+1}{n\sin^2\left(\pi/2n\right)+1}\sin\left(\dfrac{\pi}{2n}\right)\\
&=2\sin\left(\pi/2n\right)=a_n
\end{align*}
On the other hand, Reilly's inequality and (\ref{Spectrum-Gamma-2n}) yields
\begin{equation}\label{major-lambda1-F_n}
\lambda_1^{F_n}\leqslant\dfrac{1}{P(F_n)}\displaystyle\sum_{i=0}^{n+1}|H_i|^2=a_n=\lambda_1^{\Gamma_{2n}}.
\end{equation}
To complete the proof it remains to prove that $\lambda_1^{\Gamma_{2n}}\leqslant\lambda_1^{F_n}$. Let $u$ be a nonzero eigenfunction of $F_n$ associated to $\lambda_1^{F_n}$. Let $u^s$ be the function defined on $F_n$ for any $k\in\{0,\cdots,n+1\}$ by
$$ u^s(A_k)=u(s_{(OA_0)}(A_k))$$
where for any line $\mathcal{D}$, $s_{\mathcal{D}}$ denotes the reflection with respect to $\mathcal{D}$. Note that $u^s$ is still an eigenfunction associated to $\lambda_1^{F_n}$. If $u^s=u$, $u$ will be said symmetric and if $u^s=-u$, $u$ will be said antisymmetric. Let $E$ be the space of first eigenfunctions of $\lambda_1^{F_n}$, $\mathcal{S}=\{u\in E\ \ ;\ \ u^s=u\}$ and $\mathcal{A}=\{u\in E\ \ ;\ \ u^s=-u\}$. Then $E=\mathcal{S}\oplus\mathcal{A}$. In a first time assume that $\mathcal{S}\neq\{0\}$ and let $u\in\mathcal{S}\setminus\{0\}$. For any $k\in\{0,\cdots,n+1\}$, we put $u_k=u(A_k)$. In particular we have $u_1=u_{n+1}$ because $u$ is symmetric. Let $v$ defined on $\Gamma_{2n}$ by 
\begin{align*}
    v_1&=v(B_1)=v(B_{n+1})=v_{n+1}=0, \\
  v_k&=v(B_k)=u_k-u_1, \quad \textnormal{ for } k\in\{2,\cdots,n\},\\
    v_k&=v(B_k)=-v(s_{(A_1A_{n+1})}(B_k)), \quad \textnormal{ for } k\in\{n+2,\cdots,2n\}.
\end{align*}
Then using $$\sum_{i=1}^{2n}v_i=0,$$ we get 
\begin{align*}
\int_{\Gamma_{2n}}|\nabla v|^2\;d\mathcal{H}_1&\geqslant a_n\sum_{i=1}^{2n}v_i^2=2a_n\sum_{i=2}^{n}(u_i-u_1)^2\\
&=2a_n\sum_{i=2}^{n}u_i^2-4a_nu_1\sum_{i=2}^{n}u_i+2(n-1)a_nu_1^2.
\end{align*}
Since $u$ is a first eigenfunction of $F_n$, we have 
$$\sum_{i=0}^{n+1}u_i=2u_1+u_0+\sum_{i=2}^{n}u_i=0,$$
where we have used the fact that $u_1=u_{n+1}$. Then we deduce that 
\begin{align*}
\int_{\Gamma_{2n}}|\nabla v|^2\;d\mathcal{H}_1&\geqslant2a_n\sum_{i=2}^{n}u_i^2+4a_nu_1(2u_1+u_0)+2(n-1)a_nu_1^2\\
&=2a_n\sum_{i=0}^{n+1}u_i^2-2a_nu_0^2-4a_nu_1^2+4a_nu_1(2u_1+u_0)+2(n-1)a_nu_1^2\\
&=2a_n\sum_{i=0}^{n+1}u_i^2-2a_nu_0^2+4a_nu_1u_0+2(n+1)a_nu_1^2\end{align*}
On the other hand,
$$\int_{\Gamma_{2n}}|\nabla v|^2d\mathcal{H}_1=\frac{1}{a_n}\sum_{i=1}^{2n}(v_{i+1}-v_i)^2=\frac{1}{a_n}\left(2\sum_{i=2}^{n-1}(u_{i+1}-u_i)^2+4(u_2-u_1)^2\right)$$
Combining this with the previous inequality, we get 
\begin{equation}\label{Calcul-intermÃ©diaire}
\frac{1}{a_n}\sum_{i=2}^{n-1}(u_{i+1}-u_i)^2+\dfrac{2}{a_n}(u_2-u_1)^2\geqslant a_n\sum_{i=0}^{n+1}u_i^2-a_nu_0^2+2a_nu_1u_0+(n+1)a_nu_1^2.
\end{equation}
We can assume without loss of generality that $$\displaystyle\sum_{i=0}^{n+1}u_i^2=1.$$ Now since $u$ is an eigenfunction associated to $\lambda_1^{F_n}$, we have
\begin{align*}
\lambda_1^{F_n}&= \displaystyle\int_{F_n}|\nabla u|^2d\mathcal{H}_1=\dfrac{1}{a_n}\displaystyle\sum_{i=1}^n(u_{i+1}-u_i)^2+\dfrac{2}{b_n}(u_0-u_1)^2\\
&= \dfrac{1}{a_n}\displaystyle\sum_{i=2}^{n-1}(u_{i+1}-u_i)^2+\dfrac{2}{a_n}(u_2-u_1)^2+a_n(u_0-u_1)^2
\end{align*}
In this last equality we have used the fact that $u_2=u_n$, $u_1=u_{n+1}$ and $a_n=2/b_n$. Using (\ref{Calcul-intermÃ©diaire}),
\begin{align*}
\lambda_1^{F_n}&\geqslant a_n\displaystyle\sum_{i=0}^{n+1}u_i^2-a_nu_0^2+2a_nu_1u_0+(n+1)a_nu_1^2+a_n(u_0-u_1)^2\\
&=a_n\displaystyle\sum_{i=0}^{n+1}u_i^2+(n+2)a_nu_1^2 \geqslant a_n=\lambda_1^{\Gamma_{2n}}.
\end{align*}
From (\ref{major-lambda1-F_n}), we get $\lambda_1^{F_n}=\lambda_1^{\Gamma_{2n}}$. Note that this equality implies that any symmetric eigenfunction for $\lambda_1^{F_n}$ satisfies $u_1=u_{n+1}=0$. Now if $\mathcal{S}=\{0\}$, then $\mathcal{A}\neq\{0\}$. In this case we consider $u\in\mathcal{A}\setminus\{0\}$. Let $v$ defined on $\Gamma_{2n}$ by
\begin{align*}
    v_k&=v(B_k)=u_k, \quad \textnormal{ for }k\in\{2,\cdots,n\}\\
    v_k&=v(B_k)=v(s_{(A_1A_{n+1})}(B_k)), \quad \textnormal{ for }k\in\{n+2,\cdots,2n\}.
\end{align*}
Then,
\begin{equation}\label{major-lambda-1-Gamma-2n}
\left(\sum_{i=1}^{2n}v_i^2\right)^{-1}\displaystyle\int_{\Gamma_{2n}}|\nabla v|^2\;d\mathcal{H}_1=\left( \sum_{i=1}^nu_i^2\right)^{-1}\sum_{i=1}^n\dfrac{1}{a_n}(u_{i+1}-u_i)^2\geqslant\lambda_1^{\Gamma_{2n}}=a_n,
\end{equation}
where we have used the symmetry of $u$ with respect the line $(A_1A_{n+1})$. Moreover,
\begin{align*}
\lambda_1^{F_n}&=\left( \sum_{i=0}^{n+1}u_i^2\right)^{-1}\int_{F_n}|\nabla u|^2\;d\mathcal{H}_1\\
&=\sum_{i=1}^n\dfrac{1}{a_n}(u_{i+1}-u_i)^2+\dfrac{1}{b_n}(u_1-u_0)^2+\dfrac{1}{b_n}(u_{n+1}-u_0)^2\\
&=\displaystyle\sum_{i=1}^n\dfrac{1}{a_n}(u_{i+1}-u_i)^2+a_nu_{n+1}^2. \end{align*}
In this last equality we have used the fact that $u_0=0$, $u_1=-u_{n+1}$ and $b_n=2/a_n$. From (\ref{major-lambda-1-Gamma-2n}), we deduce 
$$\lambda_1^{F_n}\geqslant a_n=\lambda_1^{\Gamma_{2n}}$$
and we conclude that $\lambda_1^{F_n}=\lambda_1^{\Gamma_{2n}}$ with (\ref{major-lambda1-F_n}).
\end{proof}

\section{Reilly inequality for star-shaped graphs}
\label{graph}
The polygon considered as a graph can be viewed as having only vertices of degree 2. The curvature at one of its vertices is the sum of two unit vectors, and controlling the direction and the norm of this curvature induces a control over the angle formed by the two edges at that vertex. We subsequently establish that equality is attained when the polygon is regular, a trapezoid, a rhombus, or a so-called fake-regular polygon. We now wish to consider this type of graph :\newline

\begin{definition}
	A $n$-star-shaped graph is the union of $n$ segments $\Gamma_i = [A_0,A_i]$ of $\mathbb{R}^N$ such that there do not exist two indices $i$ and $j$ for which $\Gamma_i \subset \Gamma_j$.
\end{definition}

Away from the center, the curvature at the end of each edge is a unit vector whose orientation is defined by the edge. In the center, the curvature is the inverse of the sum of all the unit vectors mentioned above.

\vspace{1em}

\textit{Question.} Which star-shaped graphs satisfy the Reilly inequality ?

\vspace{1em}

Let $n \in \mathbb{N}$ and consider the $n$-star-shaped graph $\Gamma$ defined as the union of the edges $\Gamma_i := [A_0,A_i]$ :

\vspace{1em}
\tikzset{every picture/.style={line width=0.75pt}} 

\begin{tikzpicture}[x=0.75pt,y=0.75pt,yscale=-1,xscale=1]
	
	\draw    (86.33,49) -- (175.33,124) ;
	\draw [shift={(86.33,49)}, rotate = 40.12] [color={rgb, 255:red, 0; green, 0; blue, 0 }  ][fill={rgb, 255:red, 0; green, 0; blue, 0 }  ][line width=0.75]      (0, 0) circle [x radius= 3.35, y radius= 3.35]   ;
	\draw    (175.33,124) -- (44.33,121) ;
	\draw [shift={(44.33,121)}, rotate = 181.31] [color={rgb, 255:red, 0; green, 0; blue, 0 }  ][fill={rgb, 255:red, 0; green, 0; blue, 0 }  ][line width=0.75]      (0, 0) circle [x radius= 3.35, y radius= 3.35]   ;
	\draw [shift={(175.33,124)}, rotate = 181.31] [color={rgb, 255:red, 0; green, 0; blue, 0 }  ][fill={rgb, 255:red, 0; green, 0; blue, 0 }  ][line width=0.75]      (0, 0) circle [x radius= 3.35, y radius= 3.35]   ;
	\draw    (175.33,124) -- (96.36,167.36) ;
	\draw [shift={(96.36,167.36)}, rotate = 151.24] [color={rgb, 255:red, 0; green, 0; blue, 0 }  ][fill={rgb, 255:red, 0; green, 0; blue, 0 }  ][line width=0.75]      (0, 0) circle [x radius= 3.35, y radius= 3.35]   ;
	\draw [shift={(175.33,124)}, rotate = 151.24] [color={rgb, 255:red, 0; green, 0; blue, 0 }  ][fill={rgb, 255:red, 0; green, 0; blue, 0 }  ][line width=0.75]      (0, 0) circle [x radius= 3.35, y radius= 3.35]   ;
	\draw    (175.33,124) -- (388.33,89) ;
	\draw [shift={(388.33,89)}, rotate = 350.67] [color={rgb, 255:red, 0; green, 0; blue, 0 }  ][fill={rgb, 255:red, 0; green, 0; blue, 0 }  ][line width=0.75]      (0, 0) circle [x radius= 3.35, y radius= 3.35]   ;
	\draw    (175.33,124) -- (175.33,250) ;
	\draw [shift={(175.33,250)}, rotate = 90] [color={rgb, 255:red, 0; green, 0; blue, 0 }  ][fill={rgb, 255:red, 0; green, 0; blue, 0 }  ][line width=0.75]      (0, 0) circle [x radius= 3.35, y radius= 3.35]   ;
	
	\draw    (388.33,89) -- (457.36,77.33) ;
	\draw [shift={(459.33,77)}, rotate = 170.41] [color={rgb, 255:red, 0; green, 0; blue, 0 }  ][line width=0.75]    (10.93,-3.29) .. controls (6.95,-1.4) and (3.31,-0.3) .. (0,0) .. controls (3.31,0.3) and (6.95,1.4) .. (10.93,3.29)   ;
	\draw [color={rgb, 255:red, 155; green, 155; blue, 155 }  ,draw opacity=1 ]   (175.33,124) -- (169.33,94) ;
	\draw [color={rgb, 255:red, 155; green, 155; blue, 155 }  ,draw opacity=1 ]   (165.72,98.56) -- (387.61,61.89) ;
	\draw [color={rgb, 255:red, 155; green, 155; blue, 155 }  ,draw opacity=1 ]   (382.9,58.4) -- (388.33,89) ;
	\draw [color={rgb, 255:red, 155; green, 155; blue, 155 }  ,draw opacity=1 ]   (168.39,101.22) -- (172.39,93.89) ;
	\draw [color={rgb, 255:red, 155; green, 155; blue, 155 }  ,draw opacity=1 ]   (385.61,59.22) -- (381.28,65.89) ;
	
	\draw (177.33,127.4) node [anchor=north west][inner sep=0.75pt]    {$A_{0}$};
	\draw (98.36,170.76) node [anchor=north west][inner sep=0.75pt]    {$A_{1}$};
	\draw (66,51.4) node [anchor=north west][inner sep=0.75pt]    {$A_{3}$};
	\draw (376.33,96.4) node [anchor=north west][inner sep=0.75pt]    {$A_{4}$};
	\draw (177.33,253.4) node [anchor=north west][inner sep=0.75pt]    {$A_{5}$};
	\draw (21,121.4) node [anchor=north west][inner sep=0.75pt]    {$A_{2}$};
	\draw (425.83,86.4) node [anchor=north west][inner sep=0.75pt]    {$\tau _{4}$};
	\draw (267.33,58.4) node [anchor=north west][inner sep=0.75pt]    {$l_{4} \ $};

\end{tikzpicture}
\begin{center}
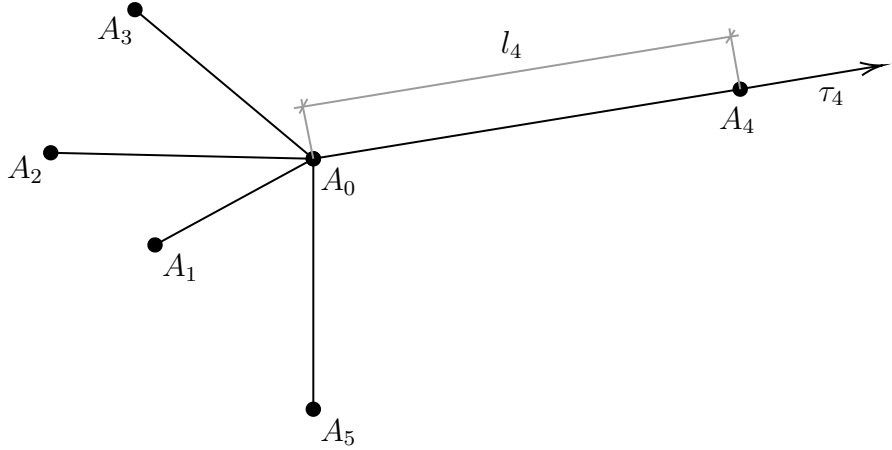
	Figure 3. A $5$-star-shaped graph in $\mathbb{R}^2$.
\end{center}

We define the length of a edge as $\ell_i := \vert \Gamma_i \vert$ and its curvature $\tau_i :=(A_i-A_0)/\ell_i$. We consider $\Gamma$ as a 1-rectifiable set and treat it as a $1$-varifold on $\mathbb{R}^N$. A simple calculation gives the first variation for all vector field $\theta$ :
\begin{align*}
	\delta \Gamma (\theta) = \sum_{i=1}^n \theta(A_i)\cdot \tau_i + \theta(A_0)\cdot \left( -\sum_{i=1}^n\tau_i\right).
\end{align*}
We define the first eigenvalue $\lambda_1>0$ of $\Gamma$ as in the Definition \ref{deflambda1} :
\begin{align*}
	\lambda_1:=\inf_{u \in C^\infty(\mathbb{R}^N) \setminus \{0\} \text{ s.t. } \sum_{i=0}^n u(A_i)=0} \left(\sum_{i=0}^nu(A_i)^2 \right)^{-1}\sum_{i=1}^n \int_{\Gamma_i}\vert \nabla_{\mathbb{R}\tau_i}u\vert^2 \; dx.
\end{align*}
Next, Reilly's inequality for star-shaped graphs is stated as follows :

\medskip\begin{proposition}\label{Reillygraphetoile}
	Let $\Gamma$ be a star-shaped graph in $\mathbb{R}^N$ as defined above. Then 
	\begin{align*}
		\lambda_1 \sum_{i=1}^n \ell_i \leqslant n + \left\vert \sum_{i=1}^n \tau_i \right\vert^2.
	\end{align*}
\end{proposition}
\begin{remark} In fact the above inequality can be rewritten in the following form :
\begin{align*}
		\lambda_1 \mathcal{H}^1(\Gamma) \leqslant \sum_{i=0}^n\vert \mathbf{H}_i \vert^2.
	\end{align*}
Indeed : 
	\begin{align*}
		\vert \mathbf{H}_0 \vert =\left\vert \sum_{i=1}^n \tau_i \right\vert, \textnormal{ and }\quad \vert \mathbf{H}_i \vert = \vert \tau_i \vert = 1.
	\end{align*}
\end{remark}
\begin{proof}
	The graph $\Gamma$ is not a priori centered, which means that $A_0$ can be different both from 0 and from its center of mass. Suppose that $\sum_{i=0}^nA_i=0$ ; its center of mass is $0$. We use the coordinate projections $x_j(X):= X \cdot e_j$ as test functions, where $(e_j)_{1\leqslant j \leqslant N}$ is the canonical basis of $\mathbb{R}^N$. By definition of $\lambda_1$, we have
	\begin{align*}
		\lambda_1 \sum_{i=0}^n(A_i \cdot e_j)^2 \leqslant \sum_{i=1}^n\int_{\Gamma_i}\vert \tau_i(\tau_i \cdot e_j)\vert^2 dx = \sum_{i=1}^n \ell_i (\tau_i \cdot e_j)^2.
	\end{align*}
	We obtain by summing over $j$ :
	\begin{align*}
		\lambda_1 \sum_{i=0}^n \vert A_i \vert^2 \leqslant \sum_{i=1}^n \ell_i \vert \tau_i \vert^2 = \sum_{i=1}^n\ell_i,
	\end{align*}
	Then, noting that $\ell_i \vert \tau_i \vert^2 = \tau_i \cdot (A_i-A_0)$, by multiplying by the sum of the lengths on each side, 
	\begin{align}\label{holderreillyetoile}
		\lambda_1\left(\sum_{i=1}^n\ell_i\right)\left(\sum_{i=0}^n \vert A_i \vert^2\right) &\leqslant \left(\sum_{i=1}^n \tau_i\cdot (A_i-A_0)\right)^2 \notag \\&\leqslant \left( \sum_{i=0}^n \vert A_i\vert^2\right)\left(\sum_{i=1}^n\vert \tau_i \vert^2 + \left\vert \sum_{i=1}^n \tau_i \right\vert^2\right).
	\end{align}
\end{proof}
We now focus on the case of equality in Reilly. We parametrize each edge $\Gamma_i$ by $\Gamma_i(t):= A_0 + t(A_i-A_0)$ for $t\in[0,1]$, note that choosing a parametrization on this interval allows the length of the segment to appear in the derivatives. Let $u$ be an eigenfunction, we show in the same manner as in Proposition \ref{Equality-Case} that there exist $n+1$ real numbers $\alpha, \beta_i$, $1 \leqslant i \leqslant n$, such that $u$ satisfies the following system of equations 
\begin{align}\label{eqgraphetoile}
	\left\{
	\begin{array}{ll}
		u\circ \Gamma_i(t)=\alpha +t \beta_i, \quad   t\in[0,1] \text{ and }  1\leqslant i \leqslant n,\\
		\lambda_1(\alpha+\beta_i)= \beta_i/\ell_i, \quad1\leqslant i  \leqslant n,\\ \displaystyle
		\lambda_1\alpha = - \sum_{i=1}^n \beta_i/\ell_i.
	\end{array}
	\right. 
\end{align}
The continuity of the eigenfunction is contained in the fact that, on each edge, at the vertex $A_0$, the function takes the value $\alpha$.

\medskip\begin{proposition}
	Let $\Gamma$ be a star-shaped graph defined as above. The graph $\Gamma$ satisfies the case of equality in the Reilly inequality \ref{Reillygraphetoile} if, and only if, $A_0=0$, all the vertices belong to the same circle of radius $1/\lambda_1$, and $A_0$ is stationary, namely $\sum_{i=1}^n \tau_i = 0$.
\end{proposition}\leavevmode

\begin{remark}
	The term "stationary" refers to graphs like this :
	
	\begin{center}
		
		\tikzset{every picture/.style={line width=0.75pt}} 
		
		\begin{tikzpicture}[x=0.75pt,y=0.75pt,yscale=-1,xscale=1]
			
			\draw    (88.17,64.04) -- (116.58,119.13) ;
			\draw [shift={(88.17,64.04)}, rotate = 62.71] [color={rgb, 255:red, 0; green, 0; blue, 0 }  ][fill={rgb, 255:red, 0; green, 0; blue, 0 }  ][line width=0.75]      (0, 0) circle [x radius= 3.35, y radius= 3.35]   ;
			\draw    (116.58,119.13) -- (54.02,119.13) ;
			\draw [shift={(54.02,119.13)}, rotate = 180] [color={rgb, 255:red, 0; green, 0; blue, 0 }  ][fill={rgb, 255:red, 0; green, 0; blue, 0 }  ][line width=0.75]      (0, 0) circle [x radius= 3.35, y radius= 3.35]   ;
			\draw [shift={(116.58,119.13)}, rotate = 180] [color={rgb, 255:red, 0; green, 0; blue, 0 }  ][fill={rgb, 255:red, 0; green, 0; blue, 0 }  ][line width=0.75]      (0, 0) circle [x radius= 3.35, y radius= 3.35]   ;
			\draw    (116.58,119.13) -- (82.17,171.54) ;
			\draw [shift={(82.17,171.54)}, rotate = 123.29] [color={rgb, 255:red, 0; green, 0; blue, 0 }  ][fill={rgb, 255:red, 0; green, 0; blue, 0 }  ][line width=0.75]      (0, 0) circle [x radius= 3.35, y radius= 3.35]   ;
			\draw [shift={(116.58,119.13)}, rotate = 123.29] [color={rgb, 255:red, 0; green, 0; blue, 0 }  ][fill={rgb, 255:red, 0; green, 0; blue, 0 }  ][line width=0.75]      (0, 0) circle [x radius= 3.35, y radius= 3.35]   ;
			\draw    (116.58,119.13) -- (147.67,174.04) ;
			\draw [shift={(147.67,174.04)}, rotate = 60.49] [color={rgb, 255:red, 0; green, 0; blue, 0 }  ][fill={rgb, 255:red, 0; green, 0; blue, 0 }  ][line width=0.75]      (0, 0) circle [x radius= 3.35, y radius= 3.35]   ;
			\draw  [color={rgb, 255:red, 155; green, 155; blue, 155 }  ,draw opacity=1 ][dash pattern={on 0.84pt off 2.51pt}] (54.02,119.13) .. controls (54.02,84.57) and (82.03,56.56) .. (116.58,56.56) .. controls (151.14,56.56) and (179.15,84.57) .. (179.15,119.13) .. controls (179.15,153.68) and (151.14,181.69) .. (116.58,181.69) .. controls (82.03,181.69) and (54.02,153.68) .. (54.02,119.13) -- cycle ;
			\draw    (147.17,65.54) -- (116.58,119.13) ;
			\draw [shift={(116.58,119.13)}, rotate = 119.72] [color={rgb, 255:red, 0; green, 0; blue, 0 }  ][fill={rgb, 255:red, 0; green, 0; blue, 0 }  ][line width=0.75]      (0, 0) circle [x radius= 3.35, y radius= 3.35]   ;
			\draw [shift={(147.17,65.54)}, rotate = 119.72] [color={rgb, 255:red, 0; green, 0; blue, 0 }  ][fill={rgb, 255:red, 0; green, 0; blue, 0 }  ][line width=0.75]      (0, 0) circle [x radius= 3.35, y radius= 3.35]   ;
			\draw    (179.15,119.13) -- (116.58,119.13) ;
			\draw [shift={(116.58,119.13)}, rotate = 180] [color={rgb, 255:red, 0; green, 0; blue, 0 }  ][fill={rgb, 255:red, 0; green, 0; blue, 0 }  ][line width=0.75]      (0, 0) circle [x radius= 3.35, y radius= 3.35]   ;
			\draw [shift={(179.15,119.13)}, rotate = 180] [color={rgb, 255:red, 0; green, 0; blue, 0 }  ][fill={rgb, 255:red, 0; green, 0; blue, 0 }  ][line width=0.75]      (0, 0) circle [x radius= 3.35, y radius= 3.35]   ;
			\draw    (289.58,120.13) -- (227.02,120.13) ;
			\draw [shift={(227.02,120.13)}, rotate = 180] [color={rgb, 255:red, 0; green, 0; blue, 0 }  ][fill={rgb, 255:red, 0; green, 0; blue, 0 }  ][line width=0.75]      (0, 0) circle [x radius= 3.35, y radius= 3.35]   ;
			\draw [shift={(289.58,120.13)}, rotate = 180] [color={rgb, 255:red, 0; green, 0; blue, 0 }  ][fill={rgb, 255:red, 0; green, 0; blue, 0 }  ][line width=0.75]      (0, 0) circle [x radius= 3.35, y radius= 3.35]   ;
			\draw    (289.58,120.13) -- (320.67,175.04) ;
			\draw [shift={(320.67,175.04)}, rotate = 60.49] [color={rgb, 255:red, 0; green, 0; blue, 0 }  ][fill={rgb, 255:red, 0; green, 0; blue, 0 }  ][line width=0.75]      (0, 0) circle [x radius= 3.35, y radius= 3.35]   ;
			\draw  [color={rgb, 255:red, 155; green, 155; blue, 155 }  ,draw opacity=1 ][dash pattern={on 0.84pt off 2.51pt}] (227.02,120.13) .. controls (227.02,85.57) and (255.03,57.56) .. (289.58,57.56) .. controls (324.14,57.56) and (352.15,85.57) .. (352.15,120.13) .. controls (352.15,154.68) and (324.14,182.69) .. (289.58,182.69) .. controls (255.03,182.69) and (227.02,154.68) .. (227.02,120.13) -- cycle ;
			\draw    (320.17,66.54) -- (289.58,120.13) ;
			\draw [shift={(289.58,120.13)}, rotate = 119.72] [color={rgb, 255:red, 0; green, 0; blue, 0 }  ][fill={rgb, 255:red, 0; green, 0; blue, 0 }  ][line width=0.75]      (0, 0) circle [x radius= 3.35, y radius= 3.35]   ;
			\draw [shift={(320.17,66.54)}, rotate = 119.72] [color={rgb, 255:red, 0; green, 0; blue, 0 }  ][fill={rgb, 255:red, 0; green, 0; blue, 0 }  ][line width=0.75]      (0, 0) circle [x radius= 3.35, y radius= 3.35]   ;
			\draw    (428.71,62.17) -- (447.58,122.13) ;
			\draw [shift={(428.71,62.17)}, rotate = 72.53] [color={rgb, 255:red, 0; green, 0; blue, 0 }  ][fill={rgb, 255:red, 0; green, 0; blue, 0 }  ][line width=0.75]      (0, 0) circle [x radius= 3.35, y radius= 3.35]   ;
			\draw    (447.58,122.13) -- (429.21,181.17) ;
			\draw [shift={(429.21,181.17)}, rotate = 107.29] [color={rgb, 255:red, 0; green, 0; blue, 0 }  ][fill={rgb, 255:red, 0; green, 0; blue, 0 }  ][line width=0.75]      (0, 0) circle [x radius= 3.35, y radius= 3.35]   ;
			\draw [shift={(447.58,122.13)}, rotate = 107.29] [color={rgb, 255:red, 0; green, 0; blue, 0 }  ][fill={rgb, 255:red, 0; green, 0; blue, 0 }  ][line width=0.75]      (0, 0) circle [x radius= 3.35, y radius= 3.35]   ;
			\draw    (447.58,122.13) -- (466.71,181.67) ;
			\draw [shift={(466.71,181.67)}, rotate = 72.19] [color={rgb, 255:red, 0; green, 0; blue, 0 }  ][fill={rgb, 255:red, 0; green, 0; blue, 0 }  ][line width=0.75]      (0, 0) circle [x radius= 3.35, y radius= 3.35]   ;
			\draw  [color={rgb, 255:red, 155; green, 155; blue, 155 }  ,draw opacity=1 ][dash pattern={on 0.84pt off 2.51pt}] (385.02,122.13) .. controls (385.02,87.57) and (413.03,59.56) .. (447.58,59.56) .. controls (482.14,59.56) and (510.15,87.57) .. (510.15,122.13) .. controls (510.15,156.68) and (482.14,184.69) .. (447.58,184.69) .. controls (413.03,184.69) and (385.02,156.68) .. (385.02,122.13) -- cycle ;
			\draw    (467.71,63.17) -- (447.58,122.13) ;
			\draw [shift={(447.58,122.13)}, rotate = 108.85] [color={rgb, 255:red, 0; green, 0; blue, 0 }  ][fill={rgb, 255:red, 0; green, 0; blue, 0 }  ][line width=0.75]      (0, 0) circle [x radius= 3.35, y radius= 3.35]   ;
			\draw [shift={(467.71,63.17)}, rotate = 108.85] [color={rgb, 255:red, 0; green, 0; blue, 0 }  ][fill={rgb, 255:red, 0; green, 0; blue, 0 }  ][line width=0.75]      (0, 0) circle [x radius= 3.35, y radius= 3.35]   ;
			
			\draw (61.02,176.42) node [anchor=north west][inner sep=0.75pt]    {$A_{1}$};
			\draw (64.33,44.4) node [anchor=north west][inner sep=0.75pt]    {$A_{3}$};
			\draw (149.67,177.44) node [anchor=north west][inner sep=0.75pt]    {$A_{7}$};
			\draw (31.67,120.73) node [anchor=north west][inner sep=0.75pt]    {$A_{2}$};
			\draw (148.67,45.84) node [anchor=north west][inner sep=0.75pt]    {$A_{4}$};
			\draw (181.15,122.53) node [anchor=north west][inner sep=0.75pt]    {$A_{5}$};
			\draw (80.33,123.18) node [anchor=north west][inner sep=0.75pt]    {$A_{0}$};
			\draw (134.06,86.73) node [anchor=north west][inner sep=0.75pt]  [font=\scriptsize]  {$\lambda _{1}^{-1}$};
			\draw (205.36,122.76) node [anchor=north west][inner sep=0.75pt]    {$A_{1}$};
			\draw (322.67,178.44) node [anchor=north west][inner sep=0.75pt]    {$A_{3}$};
			\draw (322.67,49.07) node [anchor=north west][inner sep=0.75pt]    {$A_{2}$};
			\draw (267.33,124.84) node [anchor=north west][inner sep=0.75pt]    {$A_{0}$};
			\draw (307.06,87.73) node [anchor=north west][inner sep=0.75pt]  [font=\scriptsize]  {$\lambda _{1}^{-1}$};
			\draw (408.02,182.42) node [anchor=north west][inner sep=0.75pt]    {$A_{1}$};
			\draw (470.83,44.4) node [anchor=north west][inner sep=0.75pt]    {$A_{3}$};
			\draw (405.67,42.23) node [anchor=north west][inner sep=0.75pt]    {$A_{2}$};
			\draw (468.71,185.07) node [anchor=north west][inner sep=0.75pt]    {$A_{4}$};
			\draw (421.83,113.18) node [anchor=north west][inner sep=0.75pt]    {$A_{0}$};
			\draw (464.56,87.73) node [anchor=north west][inner sep=0.75pt]  [font=\scriptsize]  {$\lambda _{1}^{-1}$};

		\end{tikzpicture}
		
	\end{center}
	\begin{center}
		Figure 4. Exemples of stationnary graphs.
	\end{center}
\end{remark}
\begin{proof}
	In the equation (\ref{holderreillyetoile}), we are in the equality case of the H\"older inequality and there are $c \in \mathbb{R}$ such that 
	\begin{align*}
\forall	\leqslant i \leqslant n\quad, 	\tau_i = cA_i, \quad 1 , \quad \textnormal{ and }  -\sum_{i=1}^n \tau_i = c A_0.
	\end{align*}
From the definition of $\tau_i$, we get $A_0=(1-c\ell_i)A_i$ for any $1 \leqslant i \leqslant n$. If $n=2$ and $A_0,A_1,A_2$ are aligned, then $\tau_1+\tau_2=0$ and $A_0=0$. Otherwise, there exists two indices $i,j$ such that $A_0,A_i,A_j$ are not aligned. The equalities $A_0=(1-c\ell_i)A_i$ and $A_0=(1-c\ell_j)A_j$ still imply that $A_0=0$. As a consequence we have $c=1/\ell_i$ and 
	\begin{align*}
		\sum_{i=1}^n \tau_i = 0.
	\end{align*}
	To conclude, we use the equality 
	$$  \lambda_1\frac{n}{c} =  \lambda_1 \sum_{i=1}^n \ell_i = n + \left\vert \sum_{i=1}^n \tau_i \right\vert^2 = n, $$
	so $ \lambda_1 = c = 1/\ell.$
	Conversely, we suppose $\Gamma$ is stationnary in $A_0$ and $\ell_i=\ell$. Let $u$ be any eigenfunction of $\Gamma$ for the eigenvalue $\lambda$ satisfying (\ref{eqgraphetoile}), then we have a system of $n+1$ equations with $n+1$ variables :
	\begin{align*}
		M_{n+1}(\lambda) \begin{pmatrix}
			\alpha\\
			\beta_1\\
			\vdots \\
			\beta_n \\
		\end{pmatrix} =
		\begin{pmatrix}
			\lambda & \lambda - 1/\ell &  \ldots & 0 \\
			\vdots & \vdots &\ddots & \vdots\\
			\lambda & 0 & \ldots   & \lambda-1/\ell \\
			\lambda & 1/\ell  &\ldots & 1/\ell \\
		\end{pmatrix}
		\begin{pmatrix}
			\alpha\\
			\beta_1\\
			\vdots \\
			\beta_n \\
		\end{pmatrix}
		=
		\begin{pmatrix}
			0 \\
			0\\
			\vdots\\
			0 \\
		\end{pmatrix}.
	\end{align*}
	We want to know all the solutions $\lambda$ of the equations $\textnormal{det}(M_{n+1}(\lambda))=0$. The calculation of this determinant gives a recurrence relation
	\begin{align*}
		\textnormal{det}(M_{n+1}(\lambda))=\frac{\lambda}{\ell}(1/\ell-\lambda)^{n-1}+(1/\ell-\lambda)\textnormal{det}(M_n(\lambda)),
	\end{align*}
	whose general solution is 
	\begin{align*}
		\textnormal{det} (M_{n+1}(\lambda)) = \frac{n\lambda}{\ell}(1/\ell-\lambda)^{n-1}+\lambda(1/\ell-\lambda)^n.
	\end{align*}
	The three solutions are $0$, $1/\ell$ and $(1+n)/\ell$. This proves that 
	$$\textnormal{Spec}(\Gamma) \subseteq \{0,1/\ell,(1+n)/\ell\}.$$ But for any values $\lambda$ in $\{0,1/\ell,(1+n)/\ell\}$, there exist an non-zero eigenfunction and $\lambda$ is an eigenvalue. Finally, 
	$$\textnormal{Spec}(\Gamma) =  \{0,1/\ell,(1+n)/\ell\},$$
	and $\lambda_1=1/\ell$ so the graph $\Gamma$ satisfies the case of equality in the Reilly inequality \ref{Reillygraphetoile}.
\end{proof}


\section{Appendix 1}

In this section, we give a proof of an argument used at then end of the proof of Proposition \ref{equalityH2}. Indeed, we have used the following result.

\medskip\begin{proposition}\label{f-constant} Let $(M^m,g)$ be a connected Riemannian manifold and $f\in L^1(M)$. If for any smooth vector field $X$ with compact support,
$$\int_M f{\rm div}(X)dv_g=0,$$
then $f$ is constant on $M$.
\end{proposition}

\begin{proof} Let $(U,\Phi)$ be a local chart of $M$ and let $(x_1,\cdots,x_m)$ be the coordinates associated to $(U,\Phi)$. Let $(g_{ij})_{1\leqslant i,j\leqslant m}=g\left(\partial/\partial x_i,\partial/\partial x_j\right)_{1\leqslant i,j\leqslant m}$, $(g^{ij})_{1\leqslant i,j\leqslant m}=(g_{ij})_{1\leqslant i,j\leqslant m}^{-1}$ and $G=\det(g_{ij})$. Let $X$ be a smooth vector field such that ${\rm supp}(X)\subset U$. For more convenience, we will work with the differential $1$-form $\omega$ associated to $X$ (i.e. for any $p\in U$ and $u\in T_pM$, $\omega_p(u)=g(X_p,u)$). In this case ${\rm div}(X)=-\delta\omega$ where $\delta$ is the codifferential. Then for any $p\in U$, we have 
$$\delta\omega=-\sum_{i=1}^m(D_{e_i}\omega)(e_i),$$
where $(e_i)_{1\leqslant i\leqslant m}$ is a $g$-orthonormal basis of $T_pM$ and $D$ is the Riemannian connection of $g$. In the local chart $(U,\Phi)$,
\begin{align}\label{int-f-delta-omega}
    \int_Mf{\rm div}(X)dv_g=\int_Uf\delta\omega dv_g=\int_{\Phi(U)}(f\circ\Phi^{-1})(\delta\omega\circ\Phi^{-1})\sqrt{G}dx_1\cdots dx_m
\end{align}
Now, there exist smooth functions such that ${\rm supp}(a_i)\subset U$ and $$\omega=\displaystyle\sum_{i=1}^ma_idx_i.$$ Let us compute $\delta\omega$. We consider $p\in U$ and $(e_i)_{1\leqslant i\leqslant m}$ be a $g$-orthonormal basis of $T_pM$. Then at $p$,
\begin{align*}
\delta\omega&=-\sum_{i=1}^m\sum_{j=1}^m D_{e_i}(a_jdx_j)(e_i)\\
&=-\sum_{i=1}^m\sum_{j=1}^me_i(a_j)dx_j(e_i)-\sum_{i=1}^m\sum_{j=1}^ma_j(D_{e_i}dx_j)(e_i)\\
&=-\sum_{j=1}^mdx_j(\nabla a_j)+\sum_{j=1}^ma_j\Delta x_j
\end{align*}
where $\nabla a_j$ and $\Delta x_j$ are respectively the gradient of $a_j$ and the Laplace-Beltrami operator of $x_j$. It is well known that in local coordinates,
$$\nabla a_j=\sum_{1\leqslant k,\ell\leqslant m}g^{k\ell}\frac{\partial a_j}{\partial x_k}\frac{\partial}{\partial x_{\ell}},$$
and 
$$\sum_{j=1}^mdx_j(\nabla a_j)=\sum_{1\leqslant j,k\leqslant m}g^{kj}\frac{\partial a_j}{\partial x_k}.$$
Moreover,
$$\Delta x_j=-\frac{1}{\sqrt{G}}\sum_{1\leqslant k,\ell\leqslant m}\frac{\partial}{\partial x_k}\left(\sqrt{G}g^{k\ell}\frac{\partial x_j}{\partial x_{\ell}}\right)=-\frac{1}{\sqrt{G}}\sum_{k=1}^m\frac{\partial}{\partial x_k}(\sqrt{G}g^{kj})$$
Reporting these two equalities in (\ref{int-f-delta-omega}), we get
\begin{align*}
\int_Uf\delta\omega dv_g&=-\sum_{1\leqslant j,k\leqslant m}\int_{\Phi(U)}(f\circ\Phi^{-1})g^{kj}\frac{\partial a_j}{\partial x_k}\sqrt{G}dx_1\cdots dx_m\\
&-\sum_{1\leqslant j,k\leqslant m}\int_{\Phi(U)}(f\circ\Phi^{-1})a_j\frac{\partial}{\partial x_k}(\sqrt{G}g^{kj})dx_1\cdots dx_m\\
&=-\sum_{1\leqslant j,k\leqslant m}\int_{\Phi(U)}(f\circ\Phi^{-1})\frac{\partial}{\partial x_k}(a_jg^{kj}\sqrt{G})dx_1\cdots dx_m\\
&=\sum_{1\leqslant j,k\leqslant m}\left\langle\frac{\partial(f\circ\Phi^{-1})}{\partial x_k},a_jg^{kj}\sqrt{G}\right\rangle.
\end{align*}
Let $\psi\in C^{\infty}_c(U)$. For any $(j,\ell)\in\{1,\cdots,m\}^2$, consider $$a_{j\ell}=\dfrac{\psi}{\sqrt{G}}g_{j\ell}, \textnormal{ and }\omega_{\ell}=\displaystyle\sum_{i=1}^ma_{i\ell}dx_i.$$ Then an easy computation shows that 
$$\int_Uf\delta\omega_{\ell} dv_g=\left\langle\frac{\partial(f\circ\Phi^{-1})}{\partial x_{\ell}},\psi\right\rangle.$$
If we assume that $$\displaystyle\int_M f{\rm div}(X)dv_g=0$$ for any smooth vector field $X$ with compact support, then for any $\psi\in C_c^{\infty}(U)$ and $\ell\in\{1,\cdots,m\}$, we have
$$\int_Uf\delta\omega_{\ell} dv_g=\left\langle\frac{\partial(f\circ\Phi^{-1})}{\partial x_{\ell}},\psi\right\rangle=0$$
It follows that $f$ is constant on $U$. The connexity of $M$ allows us to conclude that $f$ is constant on $M$.
\end{proof}
  
 \section{Appendix 2}   
 In this section we provide a MATLAB Code in order to produce the results of the computations needed in the proof of Proposition \ref{calculLambda1}.
 
 \begin{verbatim}
%-- Begining of the Code
clear all
syms lambda;
syms theta;

%-- definition of the vector containing the lenghts of the edges 
L=[2*cos(theta),1/cos(theta),2*tan(theta)*sin(theta),1/cos(theta),2*cos(theta)];

%--  Computation of the product of the matrices M_i(lambda)
I=[1,0;0,1]; % definition of the identiy matrix
M=I; % initialisation of M

for i=1:4
M=[1 , 1 ; -lambda* L(i+1), L(i+1)/L(i)-lambda*L(i+1)]*M; 
end

%------ Computation of the eigenvalues
g=simplify(det(M-I))
eqn= g==0
disp('The eigenvalues are given by:')
solve(eqn , lambda) % Solve the equation to find the eigenvalues

disp('The Matrix M with lambda=lambda_1 is given by:')
simplify(subs(M,lambda,2*cos(theta)))

disp('The Matrix M with lambda=lambda_2 has determinant equal to:')
det(subs(M,lambda,1/(cos(theta) - cos(theta)^3)))

disp('The Matrix M with lambda=lambda_2 is equal to:')
simplify(subs(M,lambda,1/(cos(theta) - cos(theta)^3)))

%------End of the Code
 \end{verbatim} 
 
 Here now is the results given by the above MATLAB code.
 
   \begin{verbatim}
   
   >> polygone_Trapeze_lambda_1
 
g =
 
(4*lambda*(lambda - 2*cos(theta))^2*(lambda*cos(theta)^3 
- lambda*cos(theta) + 1))/cos(theta)^3
 
 
eqn =
 
(4*lambda*(lambda - 2*cos(theta))^2*(lambda*cos(theta)^3 
- lambda*cos(theta) + 1))/cos(theta)^3 == 0
 
The eigenvalues are given by:
 
ans =
 
                             0
 1/(cos(theta) - cos(theta)^3)
                  2*cos(theta)
                  2*cos(theta)
 
The Matrix M with lambda=lambda_1 is given by:
 
ans =
 
[ 1, 0]
[ 0, 1]
 
The Matrix M with lambda=lambda_2 has determinant equal to:
 
ans =
 
1
 
The Matrix M with lambda=lambda_2 is equal to:
 
ans =
 
[(sin(theta)^4-sin(theta)^2+sin(theta)^6+1)/(sin(theta)^2-1)^3,
(cos(4*theta)/4+3/4)/(sin(theta)^2-1)^3]
[-(cos(4*theta)+3)/(sin(theta)^2-1)^3,
(7*sin(theta)^2-7*sin(theta)^4+sin(theta)^6-3)/(sin(theta)^2-1)^3]
 
  \end{verbatim} 
  
\bibliography{biblio}

\begin{thebibliography}{10}

\bibitem{AleDocRos}
Hil{\'a}rio Alencar, Manfredo~P. do~Carmo, and Harold Rosenberg.
\newblock On the first eigenvalue of the linearized operator of the
  {{\(r\)}}-th mean curvature of a hypersurface.
\newblock {\em Ann. Global Anal. Geom.}, 11(4):387--395, 1993.

\bibitem{AubGro2}
Erwann Aubry and Jean-Francois Grosjean.
\newblock Metric shape of hypersurfaces with small extrinsic radius or large $
  {{\lambda_1}}$.
\newblock Preprint, {arXiv}:1210.5689 [math.{DG}] (2012), 2012.

\bibitem{AubGro1}
Erwann Aubry and Jean-Fran{\c{c}}ois Grosjean.
\newblock Spectrum of hypersurfaces with small extrinsic radius or large
  {{\(\lambda_{1}\)}} in {Euclidean} spaces.
\newblock {\em J. Funct. Anal.}, 271(5):1213--1242, 2016.

\bibitem{zombie}
Matthieu Bonnivard, Romain Ducasse, Antoine Lemenant, and Alessandro Zilio.
\newblock A field-road system with a rectifiable set.
\newblock {\em IFB, 2025, to appear.}

\bibitem{CheGui}
Hang Chen and Xudong Gui.
\newblock Reilly-type inequalities for submanifolds in {Cartan}-{Hadamard}
  manifolds.
\newblock Preprint, {arXiv}:2206.11164 [math.{DG}] (2022), 2022.

\bibitem{CheWan}
Hang Chen and Xianfeng Wang.
\newblock Sharp {Reilly}-type inequalities for a class of elliptic operators on
  submanifolds.
\newblock {\em Differ. Geom. Appl.}, 63:1--29, 2019.

\bibitem{ElsIli1}
A.~El~Soufi and S.~Ilias.
\newblock An inequality of ``{Reilly}''-type for submanifolds of the hyperbolic
  space.
\newblock {\em Comment. Math. Helv.}, 67(2):167--181, 1992.

\bibitem{ElsIli2}
Ahmad El~Soufi and Sa{\"{\i}}d Ilias.
\newblock Second eigenvalue of {Schr{\"o}dinger} operators and mean curvature.
\newblock {\em Commun. Math. Phys.}, 208(3):761--770, 2000.

\bibitem{Gro2}
Jean-Fran{\c{c}}ois Grosjean.
\newblock Extrinsic upper bound for the first eigenvalue of elliptic operators
  defined on submanifolds and applications.
\newblock {\em C. R. Acad. Sci., Paris, S{\'e}r. I, Math.}, 330(9):807--810,
  2000.

\bibitem{Gro0}
Jean-Fran{\c{c}}ois Grosjean.
\newblock A {Reilly} inequality for some natural elliptic operators on
  hypersurfaces.
\newblock {\em Differ. Geom. Appl.}, 13(3):267--276, 2000.

\bibitem{Gro1}
Jean-Fran{\c{c}}ois Grosjean.
\newblock Upper bounds for the first eigenvalue of the {Laplacian} on compact
  submanifolds.
\newblock {\em Pac. J. Math.}, 206(1):93--112, 2002.

\bibitem{Hei}
Ernst Heintze.
\newblock Extrinsic upper bounds for {{\(\lambda _ 1\)}}.
\newblock {\em Math. Ann.}, 280(3):389--402, 1988.

\bibitem{IliMak}
Sa{\"{\i}}d Ilias and Ola Makhoul.
\newblock A {Reilly} inequality for the first {Steklov} eigenvalue.
\newblock {\em Differ. Geom. Appl.}, 29(5):699--708, 2011.

\bibitem{NiuXu}
Yanyan Niu and Shicheng Xu.
\newblock Total squared mean curvature of immersed submanifolds in a negatively
  curved space.
\newblock Preprint, {arXiv}:2106.01912 [math.{DG}] (2021), 2021.

\bibitem{Rei}
Robert~C. Reilly.
\newblock On the first eigenvalue of the {Laplacian} for compact submanifolds
  of {Euclidean} space.
\newblock {\em Comment. Math. Helv.}, 52:525--533, 1977.

\bibitem{RotUpa}
Julien Roth and Abhitosh Upadhyay.
\newblock Reilly-type upper bounds for the {{\(p\)}}-{Steklov} problem on
  submanifolds.
\newblock {\em Bull. Aust. Math. Soc.}, 108(3):492--503, 2023.

\bibitem{Sim}
Leon Simon.
\newblock {\em Lectures on geometric measure theory}, volume~3 of {\em Proc.
  Cent. Math. Anal. Aust. Natl. Univ.}
\newblock Australian National University, Centre for Mathematical Analysis,
  Canberra, 1983.

\bibitem{Ton}
Yoshihiro Tonegawa.
\newblock {\em Brakke's mean curvature flow. {An} introduction}.
\newblock SpringerBriefs Math. Singapore: Springer, 2019.

\end{thebibliography}
\bibliographystyle{plain}

\end{document}